\documentclass{amsart}
\usepackage{amsmath} 
\usepackage{amsthm}
\usepackage{amssymb}
\usepackage{mathrsfs}
\usepackage{hyperref}
\usepackage{color}
\usepackage{todonotes}

\newtheorem{theorem}{Theorem} [section]

\newtheorem*{claim}{Claim}
\newtheorem{lemma}[theorem]{Lemma}
\newtheorem{proposition}[theorem]{Proposition}

\newtheorem{definition}[theorem]{Definition}

\numberwithin{equation}{section}

\begin{document}
	\title[Horseshoes and Lyapunov exponents for  Banach cocycles]{Horseshoes and Lyapunov exponents for  Banach cocycles over nonuniformly hyperbolic systems }

\author{Rui Zou}
\address{School of Mathematics and Statistics, Nanjing University of Information Science and Technology, Nanjing 210044,  P.R. China}
\email{zourui@nuist.edu.cn }

\author{Yongluo Cao*}
%\address{Departament of Mathematics, Soochow University,Suzhou 215006, Jiangsu, P.R. China}

\address{Departament of Mathematics, Shanghai Key Laboratory of PMMP, East China  Normal University,	Shanghai 200062, P.R. China}
\email{ylcao@suda.edu.cn}

\thanks{* Yongluo Cao is corresponding author.  This work was partially supported by NSFC (11771317,  11790274),  Science and Technology Commission
	of Shanghai Municipality (18dz22710000).}

\date{\today}

\begin{abstract}
	Let $f$ be a $C^r$$(r>1)$ diffeomorphism of a compact Riemannian manifold $M$, preserving   an ergodic hyperbolic measure $\mu$  with positive entropy,  and let $\mathcal{A}$ be a H\"older continuous cocycle of injective bounded linear operators acting on a Banach space X.  We prove that  there is a sequence of horseshoes  for $f$ and  dominated splittings for $\mathcal{A}$ on the horseshoes, such that not only the  measure theoretic entropy  of $f$ but also the Lyapunov exponents of  $\mathcal{A}$ with respect to $\mu$ can be approximated by the topological entropy of $f$ and the Lyapunov exponents of $\mathcal{A}$ on the horseshoes, respectively.
\end{abstract}



\maketitle

\section{Introduction}
   	Let $f$ be a $C^r$$(r>1)$ diffeomorphism of a compact Riemannian manifold $M$, preserving   an ergodic hyperbolic measure $\mu$  with positive entropy,  and let $\mathcal{A}$ be a H\"older continuous cocycle of injective bounded linear operators acting on a Banach space X. If the cocycle satisfies the so called quasi-compactness condition, then there is a sequence of horseshoes  for $f$ and  dominated splittings for $\mathcal{A}$ on the horseshoes, such that not only the  measure theoretic entropy  of $f$ but also the Lyapunov exponents of  $\mathcal{A}$ with respect to $\mu$ can be approximated by the topological entropy of $f$ and the Lyapunov exponents of $\mathcal{A}$ on the horseshoes, respectively. For an explicit statement, see section 2. 

    This paper is inspired by Katok \cite{Katok}(or Katok, Mendoza \cite[Theorem S.5.9]{katok1995}) and Cao, Pesin, Zhao \cite{caozhao}. The approximation of an ergodic hyperbolic measure by horseshoes was first proved by Katok \cite{Katok}. In \cite{mendoza}, Mendoza showed that,  for a $C^2$ surface diffeomorphism, an ergodic hyperbolic SRB measure can be  approximated by a horseshoe with unstable dimension converging to 1. Avila, Crovisier and Wilkinson \cite{Avila17} stated that the horseshoe constructed by Katok \cite{Katok} also has a dominated splitting( for $Df$) and the Lyapunov exponents of the hyperbolic measure can be approximated by the exponents on the horseshoe. 
    
    In the $C^1$ setting, if  a hyperbolic measure has positive entropy and whose support admits a dominated splitting, Gelfert \cite{Gelfert16} asserted the approximation of ergodic hyperbolic measures by horseshoes. Wang, Zou and Cao \cite{Wangzoucao}  furtherly studied the horseshoe approximation of Lyapunov exponents, which is used to 
    show the arbitrarily large unstable dimension of the horseshoes.
    
   The case of $C^r$$(r>1)$ maps was considered by Chung \cite{Chung99}, Yang \cite{Yang15} and Gelfert \cite{gelfert2010}.  Cao, Pesin and Zhao \cite{caozhao} constructed  repellers such that both the entropy and the Lyapunov exponents can be approximated on the repellers. They also used this result to show the continuity of sub-additive topological  pressure, and then give a lower bound estimate of the Hausdorff dimension of a non-conformal repeller.
    
    For infinite-dimensional dynamical systems,  Lian and  Young \cite{Lian2011lyapunov} generalized Katok's results \cite{Katok} to mapping of Hilbert spaces.  They  also proved analogous results for semiflows on Hilbert spaces \cite{lian2012lyapunov}.  Kalinin and Sadovskaya \cite{Kalinin-Sadovskaya} consider the Holder continuous cocycle $ \mathcal{A}$ for  invertible bounded linear operates on Banach space of $f$ as above, and prove that the upper and lower Lyapunov exponents of $\mathcal A$ with respect to $\mu$ can be approximated in term of the norms of the return values of $\mathcal A$ on hyperbolic periodic point of $f$.

    The main result of this paper is stated in Section 2, together with some notations and preliminaries. In Section 3, we provide the proof of the main result.

%%%%%%%%%%%%%%%%%%%%%%%%%%%  Section 2 %%%%%%%%%%%%%%%%%%%%%%%%%
\section{Preliminaries}
\subsection{Cocycles and Lyapunov exponents  for cocycles}
  Let $f$ be a $C^r$$(r>1)$ diffeomorphism of a compact Riemannian manifold $M$, and $L(X)$ be the space of bounded linear operators on a Banach space X. Assume $A:M \to L(X)$ is a H\"older continuous map. The{ \em cocycle}~ over $f$ generated by $A$ is a map $\mathcal{A}:M\times \mathbb{N}\to L(X)$ defined by $\mathcal{A}(x,0)=Id,$ and $\mathcal{A}(x,n)=A(f^{n-1}x)\cdots A(fx)A(x)$. We also denote $\mathcal{A}(x,-n):=A(f^{-n}x)^{-1}\cdots A(f^{-2}x)^{-1}A(f^{-1}x)^{-1}$ for $n>0.$ Note that $\mathcal{A}(x,-n)$ is not a map in general. To be more flexible, we also denote $\mathcal{A}_x^n:=\mathcal{A}(x,n).$

   We now study some properties of Lyapunov exponents. The following version of multiplicative ergodic theorem was established by  P. Thieullen \cite{Thieullen87}, based on the work of \cite{Oseledets1968,Ruelle82,mane1983}. In order to state the multiplicative ergodic theorem,  we introduce some definitions. 
    
    Denote by $B_1$ the unit ball of X. Then for any $T\in L(X),$ We define the {\em Hausdorff measure of noncompactness~} of $T$ by 
    \[\|T\|_\kappa :=\inf\{\varepsilon>0:  T(B_1)~\text{can be covered by finite} ~\varepsilon\text{-balls}  \}.\]
    Then by the definition, we have  $\|T\|_\kappa\leq \|T\|$, and $\|\cdot\|_\kappa$ is sub-multiplicative, that is $\|T_2T_1\|_\kappa\leq\|T_2\|_\kappa\cdot\|T_1\|_\kappa$ for any $T_1,T_2\in L(X).$ Let $\mu$ be  an ergodic $f$-invariant measure on $M,$ then by the sub-additive ergodic theorem, the limits
    \begin{align*}
    \chi(\mathcal{A},\mu):=\lim\limits_{n\to +\infty}\frac{1}{n}\int \log \|\mathcal{A}_x^n\|d\mu,\\
    \kappa(\mathcal{A},\mu):=\lim\limits_{n\to +\infty}\frac{1}{n}\int \log \|\mathcal{A}_x^n\|_\kappa d\mu
    \end{align*}
    exist. We call the cocycle $\mathcal{A}$ is {\em quasi-compact~} with respect to $\mu$, if $\chi(\mathcal{A},\mu)>\kappa(\mathcal{A},\mu)$.
    
    Given two topological spaces $Y$, $Z$ and a Borel measure $\mu$ on $Y$, a map $g:Y\to Z$ is called {\em $\mu$-continuous}, if there exists a sequence of pairwise disjoint compact subsets $Y_n\subset Y$, such that $\mu(\cup_{n\geq 1}Y_n)=1$ and $g|_{Y_n}$ is continuous for every $n\geq 1$. We write $\mathbb{N}_k:=\{1,2,\cdots,k \}$ for $1\leq k <+\infty$ and $\mathbb{N}_{+\infty}:=\mathbb{N}_+$. 
    
     For a given $f$-invariant set $\Lambda,$ a splitting $X=E_1(x)\oplus\cdots\oplus E_i(x)\oplus F_i(x)$ on $\Lambda$ is called an {\em $\mathcal{A}$-invariant splitting}, if 
    $A(x)E_j(x)=E_j(fx)$ and $A(x)F_i(x)\subset F_i(fx)$ for every $x\in \Lambda, j=1,\cdots,i.$
    
   We now state a version of  multiplicative ergodic theorem used in this paper. 
\begin{theorem}[\cite{Thieullen87}]
Let $f$ be a $C^r$$(r>1)$ diffeomorphism of a compact Riemannian manifold $M$, preserving  an ergodic measure $\mu$ , and let $A:M\to L(X)$  be a H\"older continuous map such that $A(x)$ is injective for every $x\in M$. If $\mathcal{A}$ is quasi-compact with respect to $\mu$, then there exists an
	$f$-invariant set $\mathcal{R}^\mathcal{A}\subset M$ with $\mu(\mathcal{R}^\mathcal{A})=1$ such that for every $x\in \mathcal{R}^\mathcal{A}$:
	\begin{enumerate}
		\item There exist  $k_0\leq+\infty$, numbers  $\lambda(\mathcal{A},\mu)=\lambda_{1}>\lambda_{2}>\cdots >\kappa(\mathcal{A},\mu) $, finite dimensional subspaces $E_1(x),E_2(x),\cdots$ and infinite dimensional closed subspaces $F_1(x),F_2(x),\cdots$ indexed in $\mathbb{N}_{k_0}$.
		\item For any $i\in \mathbb{N}_{k_0}$, there exists an  $\mathcal{A}$-invariant splitting on $\mathcal{R}^\mathcal{A}$:
		\[X=E_1(x)\oplus\cdots\oplus E_i(x)\oplus F_i(x).\]
		\item For any $ i\in \mathbb{N}_{k_0}, \lim\limits_{n\rightarrow \pm \infty} \frac{1}{n}\log\frac{\|\mathcal{A}_x^n(u)\|}{\|u\|} = \lambda_{i},\ \forall u\in E_{i}(x)\setminus \{0\}$; \\
		$\lim\limits_{n\rightarrow +\infty} \frac{1}{n}\log\|\mathcal{A}_x^n|_{F_i(x)}\| = \lambda_{i+1},\ $ where $\lambda_{{k_0}+1}:=\kappa(\mathcal{A},\mu)$ if $k_0<+\infty.$	
		\item  For any $i\in \mathbb{N}_{k_0}$, $E_i(x),F_i(x)$ are Borel measurable $\mu$-continuous and the norms of the projection operators $\pi^i_E(x),\pi^i_F(x)$ associated with  ${X}=\big(E_1(x)\oplus\cdots\oplus E_i(x)\big)\oplus F_i(x)$ are tempered, that is,
		\begin{equation}\label{2.1}
		\lim\limits_{n\to \pm\infty}\frac{1}{n}\log \|\pi^i_\tau(f^nx)\|=0,~~\ \forall \tau=E,F.
		\end{equation} 		
	\end{enumerate}
	
\end{theorem}

   The numbers $\lambda_{1}>\lambda_{2}>\cdots$	indexed in $\mathbb{N}_{k_0}$ are called the { \em Lyapunov exponents of $\mathcal{A}$ with respect to $\mu,$} and the decomposition ${X}=E_1(x)\oplus\cdots\oplus E_i(x)\oplus F_i(x)$ is called the {\em Oseledets decomposition}. Denote $d_j:=\dim(E_j).$
 
%%%%%%%%%%%%%%%%%%%%%%%%   Main result   %%%%%%%%%%%%%%% 
   
\subsection{Main result} Recall that an ergodic $f$-invariant measure is called a \emph{hyperbolic measure} if it has no zero Lyapunov exponents for $Df$. We now state the main result of this paper.
\begin{theorem}\label{thm A}
	Let $f$ be a $C^r$$(r>1)$ diffeomorphism of a compact Riemannian manifold $M$, preserving   an ergodic hyperbolic measure $\mu$  with $h_{\mu}(f)>0$. Assume  $A:M\to L(X)$  be a H\"older continuous map such that $A(x)$ is injective for every $x\in M$ and the generated cocycle $\mathcal{A}$ is quasi-compact with respect to $\mu$. Then for any $i\in\mathbb{N}_{k_0},$ there exists $\varepsilon_i>0$, such that for any $0<\varepsilon<\varepsilon_i,$ there exists a hyperbolic horseshoe $\Lambda$ satisfying the following properties:
	\begin{enumerate}
		\item   $|h_{\text {\em top}}(f|_{\Lambda})-h_\mu(f)|<\varepsilon.$
		\item   $\Lambda$ is $\varepsilon$-close to $\text{supp}(\mu)$ in the Hausdorff distance. 
		\item   For any $f$-invariant probability measure $\nu$ supported in $\Lambda$, $D(\mu,\nu)<\varepsilon$.
		\item   There exist $m\in \mathbb{N}$ and a continuous $\mathcal{A}$-invariant splitting of $X$ on $\Lambda$ 
		\[X=E_1(x)\oplus\cdots\oplus E_i(x)\oplus F_i(x),\]
		with~$\dim(E_j(x))=d_j,\forall 1 \leq j \leq i$, such that for any $x \in \Lambda, $ we have 
		\[e^{(\lambda_j-\varepsilon)m}\|u\|\leq \|\mathcal{A}_x^m(u)\|\leq e^{(\lambda_j+\varepsilon)m}\|u\|,\quad \forall u\in E_j(x), 1\leq j\leq i, \]
		\[\|\mathcal{A}_x^m(v)\|\leq e^{(\lambda_{i+1}+\varepsilon)m}\|v\|,\quad \forall v\in F_i(x). \]
	\end{enumerate}
\end{theorem}

%%%%%%%%%%%%%%%%%%  Lyapunov norm   %%%%%%%%%%%
\subsection{ Lyapunov norm for cocycles}
   Fix any $i\in \mathbb{N}_{k_0}$, $\varepsilon>0$ and $x\in \mathcal{R}^\mathcal{A}$.  We define the \emph{Lyapunov norm for $\mathcal{A}$} $\|\cdot\|_x=\|\cdot\|_{x,i,\varepsilon}$ on ${X}$ as follows: \\ 	
   For any  $u=u_1+\cdots+u_{i+1}\in {X},$ where $u_j\in E_j(x),~ \forall j=1,\cdots ,i, $ and $ u_{i+1}\in F_i(x)$, we define
   \begin{equation}\label{2.2}
   \|u\|_x:=\sum_{j=1}^{i+1} \|u_j\|_x,
   \end{equation}
   where 
   \[\|u_j\|_x=\sum_{n=-\infty}^{+\infty}\|\mathcal{A}_x^n(u_j) \|\cdot e^{-\lambda_j n-\varepsilon |n|}, \quad\forall j=1,\cdots ,i;\] 
   and 
   \begin{equation}\label{2.3}
   \|u_{i+1}\|_x=\sum_{n=0}^{+\infty}\|\mathcal{A}_x^n(u_{i+1}) \|\cdot e^{-(\lambda_{i+1}+\varepsilon)n}.
   \end{equation}
 Then the following lemma holds.
\begin{lemma}\label{lemma 2.2}
	Let $f, \mu$ and $A$ be as in Theorem \ref{thm A}. Then for any  $i\in \mathbb{N}_{k_0}$, $\varepsilon>0$, the Lyapunov norm $\|\cdot\|_x=\|\cdot\|_{x,i,\varepsilon}$ satisfies the following properties.
	\item 	\text{\bf(\romannumeral1)} For any $x\in \mathcal{R}^\mathcal{A},$ we have
	\begin{equation}\label{2.4}
	e^{\lambda_j-\varepsilon} \|u_j\|_x  \leq \|{A}(x)u_j\|_{fx} \leq e^{\lambda_j+\varepsilon}\|u_j\|_x,\quad \forall u_j\in E_j(x),~~j=1,\cdots,i,
	\end{equation}  
	\begin{equation}\label{2.5}
	\|{A}(x)u_{i+1}\|_{fx} \leq e^{\lambda_{i+1}+\varepsilon} \|u_{i+1}\|_x,\quad \forall u_{i+1}\in F_i(x).
	\end{equation} 
	\item	\text{\bf(\romannumeral2)} There exists an $f$-invariant subset of $\mathcal{R}^\mathcal{A}$ with $\mu$-full measure (we may also denote it by $\mathcal{R}^\mathcal{A}$), and a measurable function $K(x)=K_{\varepsilon,i}(x)$ defined on $\mathcal{R}^\mathcal{A}$ such that  for any $x\in \mathcal{R}^\mathcal{A}$, $u\in {X},$  we have
	\begin{equation}\label{2.6}
	\|u\|\leq \|u\|_x\leq K(x)\|u\|,
	\end{equation}
	\begin{equation}\label{2.7}
	K(x)e^{-\varepsilon}\leq K(fx)\leq K(x)e^{\varepsilon}.
	\end{equation}
\end{lemma}
\begin{proof}
	{\bf(\romannumeral1)} We will prove the inequality 
	$$ \|{A}(x)u_j\|_{fx} \leq e^{\lambda_j+\varepsilon}\|u_j\|_x,\quad \forall u_j\in E_j(x),~~j=1,\cdots,i,$$
	the others can be proved analogously.  By the definition, we have
	\begin{align*}
	\|{A}(x)u_j\|_{fx}  & =  \sum_{n=-\infty}^{+\infty}\|\mathcal{A}_{fx}^n(\mathcal{A}_x^1(u_j)) \|\cdot e^{-\lambda_j n-\varepsilon |n|}\\	
	& =  \sum_{n=-\infty}^{+\infty}\|\mathcal{A}_x^{n+1}(u_j) \|\cdot e^{-\lambda_j (n+1)-\varepsilon |n+1|}\cdot  e^{\lambda_j +\varepsilon (|n+1|-|n|)}\\
	& \leq  e^{\lambda_j+\varepsilon}\cdot \|u_j\|_x. 
	\end{align*}
	{\bf(\romannumeral2)}  For any $u=u_1+\cdots+u_{i+1}\in {X}$, by the definition,
	$$	\|u\|\leq \|u_1\|+\cdots+\|u_{i+1}\|\leq \|u_1\|_x+\cdots+\|u_{i+1}\|_x=\|u\|_x. $$                  
	This estimates the lower bound. To estimate the upper bound, we define
	\[M_{j}(x):=\sup \left\{\frac{\|\mathcal{A}_x^n(u_j)\|}{e^{\lambda_j n+\frac{1}{2}\varepsilon|n|}\cdot\|u_j\|}: ~u_j\in E_j(x),n\in \mathbb{Z}   \right\},\quad \forall j=1,\cdots,i, \]
	\[ M_{i+1}(x)
	:=  \sup \left\{\frac{\|\mathcal{A}_x^n(u_{i+1})\|}{e^{(\lambda_{i+1}+\frac{1}{2}\varepsilon)n}\cdot\|u_{i+1}\|}:~u_{i+1}\in F_i(x),n\geq 0 \right\}.\]	  
	Then	
	\begin{equation}\label{2.8}
	\begin{split}
	&\|u\|_x \leq \sum_{j=1}^{i}\sum_{n=-\infty}^{+\infty}M_{j}(x)e^{-\frac{1}{2}\varepsilon |n|}\cdot \|u_j\| +
	\sum_{n=0}^{+\infty}M_{i+1}(x)e^{-\frac{1}{2}\varepsilon n}\cdot\|u_{i+1}\| \\
	& \leq c_0 \cdot \Big(\sum_{j=1}^{i}M_{j}(x)\cdot\|\pi^{j}_{E}(x)-\pi^{j-1}_{E}(x)\|\cdot\|u\|+M_{i+1}(x)\cdot\|\pi^{i}_{F}(x)\|\cdot\|u\|\Big)\\
	& =: M(x)\cdot\|u\|, 
	\end{split}
	\end{equation}
	where $c_0=\sum_{n=-\infty}^{+\infty}e^{-\frac{1}{2}\varepsilon |n|}$ and $\pi^{0}_{E}(x)=0$.
	\begin{claim}
		$M(x)$ is tempered on an $f$-invariant   subset of $\mathcal{R}^\mathcal{A}$ with $\mu$-full measure, that is
		\[\lim\limits_{n\to\pm\infty}\frac{1}{n}\log {M}(f^nx)=0,\quad \text{for}~ \mu\text{-a.e.} ~x\in \mathcal{R}^\mathcal{A}. \]
	\end{claim}
	\begin{proof}
		Since $\|\pi^j_{E}(x)\|$,$\|\pi^j_{F}(x)\|$ are tempered by \eqref{2.1}, it's enough to prove $M_{j}(x)$ is tempered for any $ 1\leq j\leq i+1$. Since 
		\begin{align*}
		   & M_{i+1}(x) \\
	     = &\underset{ \substack{ u_{i+1}\in F_{i+1}(x)\\ 
	     		n\geq 1
	     	} 
	     }{\sup} \left\{\frac{\|\mathcal{A}_{fx}^{n-1}( \mathcal{A}_x(u_{i+1}))\|}{e^{\lambda_{i+1} (n-1)+\frac{1}{2}\varepsilon|n-1|}\cdot\|\mathcal{A}_x(u_{i+1})\|}\cdot\frac{\|\mathcal{A}_x(u_{i+1})\|}{e^{\lambda_{i+1}+\frac{\varepsilon}{2}(|n|-|n-1|)}\cdot\|u_{i+1}\|},~1 \right\}\\
	  \leq & \max\{c_1\cdot M_{i+1}(fx),1 \},
		\end{align*}
		where $c_1=\max\limits_{x\in X}\frac{\|\mathcal{A}_x\|}{e^{\lambda_{i+1}-\frac{1}{2}\varepsilon}}$,
		we obtain 
		$$\big(\log M_{i+1}(x)-\log M_{i+1}(fx)\big)^+ \leq \log^+ \max\{c_1,\frac{1}{ M_{i+1}(fx)} \}\leq\log^+ \max\{c_1,1\},$$ 
		then we conclude $M_{i+1 }(x)$ is tempered on a subset of  full measure by \cite[Lemma \uppercase\expandafter{\romannumeral3}.8]{mane1983}. The result for  $ 1\leq j\leq i$ is obtained similarly.	
	\end{proof} 
	Let
	\[K(x):= \sum_{n\in\mathbb{Z}} {M}(f^nx)e^{-\varepsilon|n|}, \quad \forall x\in\mathcal{R}^\mathcal{A}, \]
	then by \cite[Lemma 3.5.7]{barreira000pesin}, $K(x)$ satisfies \eqref{2.6} and \eqref{2.7}. This completes the proof.		
\end{proof}
%%%%%%%%%%%%%%%%%%%%%%%%%%
 	
   Fix any  $i\in \mathbb{N}_{k_0}$. Then by Lusin's theorem, for any $\delta>0,$ there exists a compact subset $\mathcal{R}^\mathcal{A}_{\delta}\subset \mathcal{R}^\mathcal{A}$ such that $\mu(\mathcal{R}^\mathcal{A}_{\delta})>1-\delta$ and $K(x)$ is continuous on $\mathcal{R}^\mathcal{A}_\delta$. Denote \begin{equation}\label{2.9}
   l=l_\delta=\sup\limits\{K(x):x\in\mathcal{R}^\mathcal{A}_{\delta}\}.
   \end{equation}

   Since the Oseledets decomposition is $\mu$-continuous, we may assume   the Oseledets decomposition is continuous on  $\mathcal{R}^\mathcal{A}_{\delta}$.
   
   Let $\widetilde{{X}}$ be the collection of norms on ${X}$ which are equivalent to $\|\cdot\|$, 
   then $\widetilde{{X}}$ is a metric space with respect to the metric
   \[\widetilde{D}(\varphi,\psi)=\sup\limits_{u\in {X}\setminus\{0\}}\frac{|\varphi(u)-\psi(u)|}{\|u\|},\quad \forall \varphi,\psi\in\widetilde{{X}}.  \]
   We claim that the function  $x\mapsto \|\cdot\|_x$   is continuous on $\mathcal{R}^\mathcal{A}_{\delta}$. Indeed,  for any $x\in \mathcal{R}^\mathcal{A}_{\delta},$  $u=u_1+\cdots+u_{i+1}\in {X},$ where $u_j\in E_j(x),~ \forall j=1,\cdots ,i, $ and $ u_{i+1}\in F_i(x)$, we define
   \[\|u\|_{k,x}:=\sum_{j=1}^{i}\sum_{n=-k}^{k}\|\mathcal{A}_x^n(u_{j}) \|\cdot e^{-\lambda_j n-\varepsilon |n|}+\sum_{n=0}^{k}\|\mathcal{A}_x^n(u_{i+1}) \|\cdot e^{-(\lambda_{i+1}+\varepsilon)n},\]
   then by \eqref{2.8}, 
   \begin{align*}
 	|\|u\|_x-\|u\|_{k,x}| & \leq \sum_{|n|>k}e^{-\frac{1}{2}\varepsilon |n|}\cdot M(x)\cdot\|u\|                      \\
	                      & \leq \frac{2e^{-\frac{\varepsilon}{2}k}}{e^{\frac{\varepsilon}{2}}-1}\cdot K(x)\cdot\|u\| \\
	                      & \leq \frac{2l\cdot e^{-\frac{\varepsilon}{2}k}}{ (e^{\frac{\varepsilon}{2}}-1)}\cdot\|u\|,
   \end{align*}
   that is,
   \[\widetilde{D}\big(\|\cdot\|_x,~\|\cdot\|_{k,x}\big)\leq \frac{2l\cdot e^{-\frac{\varepsilon}{2}k}}{ (e^{\frac{\varepsilon}{2}}-1)}, \]
     which implies $\|\cdot\|_{k,x}\to \|\cdot\|_x$ uniformly for $x\in \mathcal{R}^\mathcal{A}_{\delta}$   as $k\to +\infty.$ Since $x\mapsto \|\cdot\|_{k,x}$ is continuous,  we have $x\mapsto \|\cdot\|_x$   is continuous on $\mathcal{R}^\mathcal{A}_{\delta}$. This proves the claim.
%todo constant dimension E F
   
%%%%%%%%%%%%%%%%%%%%%%%%%%%      
\subsection{Regular neighborhoods}
   Let $f$ be a  $C^r \, (r >1) \, $ diffeomorphism of a compact Riemannian $d$-dimensional manifold $M$, preserving  an  ergodic hyperbolic measure $\mu$. Then by Oseledets multiplicative ergodic theorem \cite{Oseledets1968}, there exists an $f$-invariant set $\mathcal{R}^{Df}$ with $\mu$-full measure, a number $\chi>0$ and a $Df$-invariant decomposition $TM=E^u\oplus E^s$ on $\mathcal{R}^{Df}$ such that for any $x\in\mathcal{R}^{Df}$, we  have
   \[\lim\limits_{n\rightarrow \pm \infty} \frac{1}{n}\log\|D_xf^n(u)\| > \chi,\ \forall u\in E^u(x)\setminus \{0\},\]
   \[\lim\limits_{n\rightarrow \pm \infty} \frac{1}{n}\log\|D_xf^n(v)\| < -\chi,\ \forall v\in E^s(x)\setminus \{0\}.\]
   
  Denote $d_u=\dim(E^u), d_s=\dim(E^s)$,  and denote by $B(0,r)$ the standard Euclidean $r$-ball in $\mathbb{R}^d$ centered at $0.$ We now introduce some properties of regular neighborhoods, see \cite{pesin76} for the proofs.
   \begin{theorem}[Pesin]\label{thm 2.3}
   	   	Let $f$ be a $C^r(r>1)$ diffeomorphism of a compact Riemannian manifold $M$, $\mu$ be an ergodic hyperbolic measure. Then for any $\varepsilon>0,$  the following properties hold.
   	   	\begin{enumerate}
   	   		\item  There exists a measurable function $l^\prime:\mathcal{R}^{Df}\to [1,\infty)$  and a collection of embeddings $\Psi_x:B(0,l^\prime(x)^{-1})\to M$ for $x\in \mathcal{R}^{Df}$ such that $\Psi_x(0)=x$, $l^\prime(x)e^{-\varepsilon}\leq l^\prime(fx)\leq l^\prime(x)e^{\varepsilon}$, and the preimages $\widetilde{E}^j(x)=(D_0\Psi_x)^{-1}E^j(x)$ are orthogonal in $\mathbb{R}^d$, where $j=u,s$.
   	   		\item  If $\widetilde{f_x}=\Psi_{fx}^{-1}\circ f\circ \Psi_x:B(0,l^\prime(x)^{-1})\to M,$ then there exist $A_u\in GL(d_u,\mathbb{R})$ and $A_s\in GL(d_s,\mathbb{R})$ such that $D_0(\widetilde{f_x})=\text{\em diag}(A_u,A_s)$ and
   	   		\[ \|A_u^{-1} \|^{-1}\geq e^{\chi-\varepsilon},~\|A_s\|\leq e^{-\chi+\varepsilon}. \]
   	   		\item  For any $a,b\in B(0,l^\prime(x)^{-1})$,
   	   		\[\|D_a(\widetilde{f_x})-D_b(\widetilde{f_x}) \|, \|D_a(\widetilde{f_x}^{-1})-D_b(\widetilde{f_x}^{-1}) \|\leq l^\prime(x)|a-b|^{r-1}.  \]
   	   		\item  There exists a constant $0<c_1<1$  such that 
   	   		\[\|D(\Psi_x)\|\leq c_1^{-1}, \|D(\Psi_x^{-1})\|\leq l^\prime(x).  \]
   	   		So for any $a,b\in B(0,l^\prime(x)^{-1})$,
   	   		\[c_1\cdot d(\Psi_x(a),\Psi_x(b))\leq |a-b|\leq l^\prime(x)\cdot d(\Psi_x(a),\Psi_x(b)).  \]
   	   	\end{enumerate} 
   \end{theorem}
   The set $\mathcal{N}(x):=\Psi_x(B(0,l^\prime(x)^{-1}))$ is called a {\em regular neighborhood} of $x$. Let $r(x)$ be the radius of maximal ball contained in $\mathcal{N}(x)$, then Theorem \ref{thm 2.3} implies $r(x)\geq l^\prime(x)^{-2}.$ By Lusin's theorem, for any $\delta>0,$ there exists a compact subset $\mathcal{R}^{Df}_{\delta}\subset \mathcal{R}^{Df}$ with  $\mu(\mathcal{R}^{Df}_{\delta})>1-\delta$ such that  $x\mapsto \Psi_x$, $l^\prime(x)$ and the Oseledets splitting  $T_xM=E_x^u\oplus E_x^s$  vary continuously on $\mathcal{R}^{Df}_\delta$. Denote $$l^\prime=l^\prime_\delta=\max\limits\{l^\prime(x):x\in\mathcal{R}^{Df}_{\delta}\}.$$
   
   For $x\in\mathcal{R}^{Df}$, using chart $\Psi_x,$ we can trivialize the tangent bundle over $\mathcal{N}(x)$ by identifying $T_{\mathcal{N}(x)}M\equiv\mathcal{N}(x)\times\mathbb{R}^d.$ For any $y\in \mathcal{N}(x), u\in T_yM,$ we can use the identification to {\em translate} the vector u to a corresponding vector $\bar{u}\in \mathbb{R}^d,$ that is, $\bar{u}=D_y(\Psi_x^{-1})u$.  We then define  
   $$\|u\|_y^\prime:=\|D_y(\Psi_x^{-1})u\|.$$
   This defines a  norm on $T_{\mathcal{N}(x)}M$.   For $v\in T_xM,$  Let $v_y\in T_yM$ be the vector translated from v, that is, $v_y=D_a(\Psi_x)D_x(\Psi_x^{-1})v$, where $a=\Psi_x^{-1}(y)$. Then $\|v\|^\prime_x=\|v_y\|^\prime_y$. 
   Thus without confusion, we may identify
   \begin{equation}\label{2.9.1}
     \|v\|^\prime_x=\|D_x(\Psi_x^{-1})v\|=\|v\|^\prime_y.
   \end{equation}
       By translating the splitting $T_xM=E^u(x)\oplus E^s(x)$, we define a new splitting $T_yM=E^u(y)\oplus E^s(y)$, where $E^u(y)=D_{a}(\Psi_x)\widetilde{E}^u(x)$ and $a=\Psi_x^{-1}(y)$(and similarly for $E^s(y)$).
   
    Using the identification $T_{\mathcal{N}(x)}M\equiv\mathcal{N}(x)\times\mathbb{R}^d,$ without confusion, we also identify $D_yf=D_{b}(\widetilde{f_x})$( where $b=\Psi_{fx}^{-1}(fy)$) and identify $T_yM=E^u(y)\oplus E^s(y)$ with $\mathbb{R}^d=\widetilde{E}^u(x)\oplus \widetilde{E}^s(x)$. Then by Theorem \ref{thm 2.3}, 
   \begin{equation}\label{2.10}
      \|D_xf(u)\|^\prime_{fx}\geq e^{\chi-\varepsilon}\|u\|^\prime_{x}, \quad \forall u\in E^u(x),
   \end{equation}
   \begin{equation}\label{2.11}
   \|D_xf(v)\|^\prime_{fx}\leq e^{-\chi+\varepsilon}\|v\|^\prime_{x}, \quad \forall v\in E^s(x),
   \end{equation} 
    \begin{equation}\label{2.12}
   c_1\|u\|\leq\|u\|^\prime_{y}\leq l^\prime(x)\|u\|, \quad \forall y\in \mathcal{N}(x),u\in T_yM.
   \end{equation}

%   For any $\varepsilon>0,$ there exist Lyapunov norms $\{|\cdot|_x:x\in\mathcal{R}^{Df} \}$ for $Df$ such that the following properties hold(for more details see \cite{barreira000pesin}):
%   \begin{enumerate}
%   	\item For any $x\in \mathcal{R}^{Df}$, we have
%   	      \[|Df_x(u)|_{fx}\geq e^{\chi-\varepsilon}|u|_x,\quad \forall u\in E^u(x),    \]
%   	      \[|Df_x(v)|_{fx}\leq e^{-\chi+\varepsilon}|v|_x,\quad \forall v\in E^s(x).    \]
%   	\item  There exists a measurable function $\widetilde{K}(x)$  such that  for any $x\in \mathcal{R}^{Df}$, $u\in T_xM,$  we have   	
%   	       \begin{equation}\label{2.10}
%   	       \frac{1}{\sqrt{d}}\|u\|\leq |u|_x\leq \widetilde{K}(x)\|u\|,
%   	       \end{equation}	
%   	       \[\widetilde{K}(x)e^{-\varepsilon}\leq \widetilde{K}(fx)\leq \widetilde{K}(x)e^{\varepsilon}.\]
%   \end{enumerate}

%\subsection{Hyperbolic measure}
%Let $f$ be a  $C^r~ (r>1)$ diffeomorphism of a compact Riemannian $d$-dimensional manifold $M$, an $f$-invariant measure $\mu$ is called a \emph{hyperbolic measure} for $f$ if for $\mu$ almost every point in $M$, its Lyapunov exponents of $f$ are non-zero. 
%%%%%%%%%%%%%%%%%%%%%%%%%%%%%%%%
\subsection{$(\rho,\beta,\gamma)$-rectangles}
    Let $M,f,\mu$ be as above, let $d_u=\dim(E^u),d_s=\dim(E^s)$, then $d_u+d_s=d.$ Let $I=[-1,1]$,   we say $R(x)\subset M$ is a \emph{rectangle} in $M$ if there exists a $C^1$ embedding $\Phi_x:I^{d}\to M$ such that $\Phi_x(I^{d})=R(x)$ and  $\Phi_x(0)=x$. A set $\widetilde{H}$ is called an  \emph{admissible $u$-rectangle in $R(x)$}, if there exist $0<\lambda<1,$ $C^1$ maps $\phi_1,\phi_2:I^{d_
    u}\to I^{d_s}$ satisfying $\|\phi_1(u)\|\geq \|\phi_2(u)\|$ for $u\in I^{d_u}$ and $\|D\phi_i\|\leq \lambda$ for $i=1,2,$ such that  $\widetilde{H}=\Phi_x(H)$, where
      $$H=\{(u,v)\in I^{d_u}\times I^{d_s}: v=t\phi_1(u)+(1-t)\phi_2(u), 0\leq t\leq1 \}.$$
   Similarly define an \emph{admissible $s$-rectangle in $R(x)$}.
\begin{definition}\label{def 2.4}
    	Given  $f:M\to M$ and $\Lambda\subset M$ compact, we say that $R(x)$ is a \emph{$(\rho,\beta,\gamma)$-rectangle} of $\Lambda$ for $\rho>\beta>0,\gamma>0$, if there exists $\lambda=\lambda(\rho,\beta,\gamma)$ satisfying:
    	\begin{enumerate}
    		\item  $x\in \Lambda, B(x,\beta)\subset \text{\em int}~R(x)$  and $\text{\em diam}(R(x))\leq \rho/3$.
    		\item If $z,f^mz\in \Lambda\cap B(x,\beta) $ for some $m>0$, then the connected component $C\big(z,R(x)\cap f^{-m}R(x)\big)$ of $R(x)\cap f^{-m}R(x)$ containing $z$ is an admissible $s$-rectangle in $R(x)$, and $f^mC\big(z,R(x)\cap f^{-m}R(x)\big)$ is an admissible $u$-rectangle in $R(x)$.
    		\item $\text{\em diam}~f^kC\big(z,R(x)\cap f^{-m}R(x)\big)\leq \rho\cdot e^{-\gamma\min\{k,m-k\}}$, for $0\leq k\leq m$.
    	\end{enumerate}
\end{definition}
 The following lemma is a simplified statement of  \cite[Theorem S.4.16]{katok1995}. 
\begin{lemma}\label{thm 2.5}
	Let $f$ be a $C^r(r>1)$ diffeomorphism of a compact Riemannian manifold $M$, $\mu$ be an ergodic hyperbolic measure. Then for any  $\rho>0,$ $\delta>0,$ there exists a constant $\beta=\beta(\rho,\delta)>0$, such that  for any $x\in \mathcal{R}_\delta^{Df}$, there exists a  $(\rho,\beta,\frac{\chi}{2})$-rectangle $R(x)$.
\end{lemma}

%%%%%%%%%%%%%%%%%%%%%

\section{Proof of Theorem \ref{thm A}}
  %todo 过渡句
  In this section, we will give the proof of Theorem \ref{thm A}. We begin with estimating the growth of vectors in certain invariant cones. 
  \subsection{Invariant cones}
  	Let $f$ be a $C^r$ diffeomorphism of a compact Riemannian manifold $M$ with $r>1$,  and $\mu$ be an ergodic hyperbolic measure for $f$ with $h_{\mu}(f)>0$. Assume  $A:M\to L(X)$  be a H\"older continuous map such that $A(x)$ is injective for  every $x\in M$  and $\lambda(\mathcal{A},\mu)>\kappa(\mathcal{A},\mu)$.  Then by the multiplicative ergodic theorem stated in subsection 2.1, there are Lyapunov exponents $\lambda_1>\lambda_2>\cdots$  indexed in $\mathbb{N}_{k_0}$ for some $k_0\leq +\infty.$ 
  	
  	Fix  any $i\in \mathbb{N}_{k_0}$, and let $\varepsilon_i:=\min\{\frac{1}{2}\chi\alpha,\frac{\chi}{4},\frac{\chi(r-1)}{2r},\frac{\lambda_1-\lambda_{2}}{8},\cdots,\frac{\lambda_{i}-\lambda_{i+1}}{8} \}$.
  	For any $0<\varepsilon< \varepsilon_i, x\in\mathcal{R}^\mathcal{A},$ let $\|\cdot\|_x=\|\cdot\|_{x,i,\varepsilon}$ be the Lyapunov norm on $X$ defined in \eqref{2.2}.  For any $1\leq j\leq i$, we have the Oseledets decomposition $X=H_j(x)\oplus F_j(x)$, where $H_j(x)=E_1(x)\oplus\cdots\oplus E_j(x)$. For any $u\in X,$ let $u=u_H+u_F$, where $u_H\in H_j(x),$ $u_F\in F_j(x)$, we  consider two cones:
       \[U_j(x,\theta):=\{u\in X:\|u_F\|_x\leq \theta \|u_H\|_x  \},  \]
       \[V_j(x,\theta):=\{u\in X:\|u_H\|_x\leq \theta \|u_F\|_x  \}.  \]	
	 For any $\delta>0,$ a sequence $(x_n)_{n\in\mathbb{Z}}$ is called  a {\em $\rho$-pseudo-orbit of $f^m$ in $\mathcal{R}^\mathcal{A}_\delta$} for some $ m \in \mathbb{N}$, if 
    $x_n,f^m(x_n)\in \mathcal{R}^\mathcal{A}_\delta$ and $d(f^mx_n,x_{n+1})\leq \rho$ for any $n\in\mathbb{Z}$. Let $l=l_\delta$ be as in \eqref{2.9}. We have the following lemma.

%todo  说句话
\begin{lemma}\label{lemma 3.1}
    For any $0<\varepsilon<\varepsilon_i,$ $\delta>0$,  there exist $\rho_0>0,  ~\theta_0>0$  such that for any $0<\rho<\rho_0 $, $1\leq j \leq i$, any $\rho$-pseudo-orbit $(x_n)_{n\in\mathbb{Z}}$ of $f^m$ in $\mathcal{R}^\mathcal{A}_\delta$ with $m\geq \frac{2\log l}{\varepsilon}$,  and for any $y\in M$ with $d(f^kx_{n},f^{k}(f^{nm}y))\leq \rho\cdot e^{-\frac{\chi}{2}\min\{k,m-k\}}$,~$k=0,\cdots,m$, we can find $\eta=\eta(\varepsilon,\delta,m)\in(0,1)$ such that
        \begin{enumerate}
        	\item $ \mathcal{A}_{f^{nm+k}y}^mU_j(f^kx_{n},\theta_0)\subset U_j(f^kx_{n+1},\eta\theta_0),$   and 
        	$$e^{(\lambda_j-4\varepsilon)m}\|u\|\leq\|\mathcal{A}_{f^{nm+k}y}^m(u)\|\leq e^{(\lambda_1+4\varepsilon)m}\|u\|, \forall u\in U_j(f^kx_n,\theta_0).$$
        	\item  $\mathcal{A}_{f^{nm+k}y}^{-m}V_j(f^kx_{n},\theta_0)\subset V_j(f^kx_{n-1},\eta\theta_0),$ and 
        	$$\|\mathcal{A}_{f^{(n-1)m+k}y}^m(v)\|\leq e^{(\lambda_{j+1}+4\varepsilon)m}\|v\|, \forall v\in \mathcal{A}_{f^{nm+k}y}^{-m}V_j(f^kx_{n},\theta_0).$$ 
        \end{enumerate}	
\end{lemma}	

\begin{proof}
	{\bf(\romannumeral1)} For simplicity of notations, we prove $\mathcal{A}_y^mU_j(x_{0},\theta_0)\subset U_j(x_{1},\eta\theta_0)$ and $ e^{(\lambda_{j}-4\varepsilon)m}\|u\|\leq\|\mathcal{A}_y^m(u)\|\leq e^{(\lambda_1+4\varepsilon)m}\|u\|, \forall u\in U_j(x_{0},\theta_0), 1\leq j\leq i.$ 
	
	For any fixed $0<\varepsilon<\varepsilon_i,$ $\delta>0$, denote $\theta_0:=e^{\lambda_{i+1}-\lambda_1}(e^\varepsilon-1)<e^\varepsilon-1$, and \\${\eta}_0:=\max \{e^{\lambda_{j+1}-\lambda_j+4\varepsilon}: 1\leq j\leq i\}$.   We have the following claim.
	\begin{claim}
		There exist  $\widetilde{\rho}_0>0,$ such that for any $\theta_0\leq \theta\leq \eta_0^{-1}\theta_0^{-1}, $ $1\leq j \leq i,$ and for any $x_{0},f^mx_{0}\in \mathcal{R}^\mathcal{A}_\delta$, $y\in M$ with  $d(f^kx_{0},f^{k}y)\leq \rho\cdot e^{-\frac{\chi}{2}\min\{k,m-k\}}$~for~$k=0,\cdots,m$ and $0<\rho<\widetilde{\rho}_0$, 
		we have  
	\begin{equation}\label{3.0.1}
		A(f^{k}y)U_j(f^{k}x_{0},\theta)\subset U_j(f^{k+1}x_{0},{\eta}_0\theta), \quad \forall ~0\leq k\leq m-1.
	\end{equation}
	\end{claim}
	\begin{proof}[Proof of the Claim]
		For any $0\leq k\leq m-1, ~u=u_H+u_F\in U_j(f^{k}x_{0},\theta)$, by \eqref{2.4} and \eqref{2.5}, we have
		\begin{equation}\label{3.1}
		\|A(f^kx_{0})u_H\|_{f^{k+1}x_{0}}\geq e^{\lambda_j-\varepsilon}\|u_H\|_{f^{k}x_{0}},
		\end{equation}
		\begin{equation}\label{3.2}
		\|A(f^kx_{0})u_F\|_{f^{k+1}x_{0}}\leq e^{\lambda_{j+1}+\varepsilon}\|u_F\|_{f^{k}x_{0}}.
		\end{equation}
		Let $w=(A(f^{k}y)-A(f^kx_{0}))u=w_H+w_F, $ where $w_H\in H_j(f^{k+1}x_{0}),$ $w_F\in F_j(f^{k+1}x_{0})$,  then 
		\begin{equation}\label{3.2.0}
		A(f^{k}y)u=w+A(f^kx_{0})u. 
		\end{equation} 
		By \eqref{2.6}, \eqref{2.7} and \eqref{2.9},
		\begin{equation}\label{3.3}
		\begin{split}
		\|w_H\|_{f^{k+1}x_{0}}\leq \|w\|_{f^{k+1}x_{0}} & \leq K(f^{k+1}x_{0})\|A(f^{k}y)-A(f^kx_{0})\|\cdot\|u\| \\
		& \leq le^{\varepsilon\min\{k+1,m-k-1 \}}\cdot c_0\rho^\alpha e^{-\frac{\chi}{2}\alpha\min\{k,m-k \}}\|u\|\\
		& \leq  c_0le^\varepsilon\rho^\alpha  e^{(\varepsilon-\frac{\chi}{2}\alpha)\min\{k,m-k \}}\|u\|_{f^{k}x_{0}}\\
		& \leq (1+\theta)c_0le^\varepsilon\rho^\alpha \|u_H\|_{f^{k}x_{0}},
		\end{split}
		\end{equation}
		since $\varepsilon-\frac{\chi}{2}\alpha<0.$ Similarly, 
		\begin{equation}\label{3.4}
		\|w_F\|_{f^{k+1}x_{0}}\leq (1+\theta)c_0le^\varepsilon\rho^\alpha\|u_H\|_{f^{k}x_{0}}.
		\end{equation}
		Let $$A(f^{k}y)u=(A(f^{k}y)u)_H+(A(f^{k}y)u)_F,$$ where $(A(f^{k}y)u)_H\in H_j(f^{k+1}x_{0}),$ $(A(f^{k}y)u)_F\in F_j(f^{k+1}x_{0})$. Then by \eqref{3.1}, \eqref{3.2.0} and \eqref{3.3}, for any $\theta_0\leq \theta\leq \eta_0^{-1}\theta_0^{-1},$
		\begin{equation}\label{3.5}
		\begin{split}
		\|(A(f^{k}y)u)_H\|_{f^{k+1}x_{0}} & \geq \|A(f^kx_{0})u_H\|_{f^{k+1}x_{0}}-\|w_H\|_{f^{k+1}x_{0}}\\
		& \geq e^{\lambda_j-\varepsilon}\|u_H\|_{f^{k}x_{0}}-(1+\theta)c_0le^\varepsilon\rho^\alpha \|u_H\|_{f^{k}x_{0}}\\
		& \geq e^{\lambda_j-2\varepsilon}\|u_H\|_{f^{k}x_{0}}
		\end{split}
		\end{equation}
		if $\rho$ is small enough. Similarly, by \eqref{3.2}, \eqref{3.2.0}  and \eqref{3.4},       
		\begin{equation}\label{3.5.1}
		\begin{split}
	  \|(A(f^{k}y)u)_F\|_{f^{k+1}x_{0}}
		\leq &  ~ \|A(f^kx_{0})u_F\|_{f^{k+1}x_{0}}+\|w_F\|_{f^{k+1}x_{0}}\\
		\leq &  ~ e^{\lambda_{j+1}+\varepsilon}\|u_F\|_{f^{k}x_{0}}+(1+\theta)c_0le^\varepsilon\rho^\alpha\|u_H\|_{f^{k}x_{0}}\\
		\leq &  ~ \theta e^{\lambda_{j+1}+2\varepsilon}\|u_H\|_{f^{k}x_{0}},
		\end{split}
		\end{equation} 
		if $\rho$ is small enough. Thus 
		$$\|(A(f^{k}y)u)_F\|_{f^{k+1}x_{0}}\leq {\eta}_0\theta \|(A(f^{k}y)u)_H\|_{f^{k+1}x_{0}},~ \forall \theta_0\leq \theta\leq \eta_0^{-1}\theta_0^{-1},$$
		that is, $A(f^{k}y)U_j(f^{k}x_{0},\theta)\subset U_j(f^{k+1}x_{0},{\eta}_0\theta)$.   
	\end{proof}
	The Claim implies 
	\begin{equation*}
	\mathcal{A}_y^mU_j(x_{0},\theta)\subset U_j(f^mx_{0},{\eta}_0\theta),~\forall \theta_0\leq\theta\leq \eta_0^{-1}\theta_0^{-1},1\leq j \leq i.
	\end{equation*}
	Since, by subsection 2.3, the Lyapunov norm and the Oseledets decomposition are uniformly continuous on the compact set $\cup_{k=1}^mf^k\mathcal{R}^\mathcal{A}_\delta$ , 
	there exists ${\eta}_0<\eta<1$ such that
	\begin{equation}\label{3.5.7}
	U_j(f^mx_{0},\eta_0\theta)\subset U_j(x_{1},\eta\theta), \quad \forall d(f^mx_{0},x_{1})\leq \rho, ~\theta_0\leq\theta\leq \eta_0^{-1}\theta_0^{-1},
	\end{equation}
	if $\rho $ is small enough. Hence for any $1\leq j \leq i,$ we have
	 \begin{equation}\label{3.5.8}
	  \mathcal{A}_y^mU_j(x_{0},\theta)\subset U_j(x_{1},\eta\theta),~\forall \theta_0\leq\theta\leq \eta_0^{-1}\theta_0^{-1}.
	 \end{equation}
	   Moreover,  for any $0\leq k\leq m-1, ~u\in U_j(f^{k}x_{0},\theta_0)$, it follows from \eqref{3.5} that  
	\begin{align*}
	 \|A(f^{k}y)u\|_{f^{k+1}x_{0}} & \geq \|(A(f^{k}y)u)_H\|_{f^{k+1}x_{0}} \\
	                               & \geq e^{\lambda_{j}-2\varepsilon}\|u_H\|_{f^{k}x_{0}} \\
                               	   & \geq (1+\theta_0)^{-1}e^{\lambda_{j}-2\varepsilon}\|u\|_{f^{k}x_{0}}\\
	                               & \geq e^{\lambda_{j}-3\varepsilon}\|u\|_{f^{k}x_{0}}.
	\end{align*}	
	Therefore, for any $1\leq j\leq i, ~u\in U_j(x_0,\theta_0)$,  by \eqref{3.0.1} and \eqref{2.6}, we conclude
	$$\|\mathcal{A}_y^m(u)\|\geq\frac{1}{l}\|\mathcal{A}_y^m(u)\|_{f^mx_{0}}\geq \frac{1}{l}e^{(\lambda_{j}-3\varepsilon)m}\|u\|_{x_{0}}\geq e^{(\lambda_{j}-4\varepsilon)m}\|u\|.$$
	Similar to \eqref{3.5}, we can also get
	\begin{equation*}
	\|(A(f^{k}y)u)_H\|_{f^{k+1}x_{0}}  \leq \|A(f^kx_{0})u_H\|_{f^{k+1}x_{0}}+\|w_H\|_{f^{k+1}x_{0}}\leq e^{\lambda_1+2\varepsilon}\|u_H\|_{f^{k}x_{0}}.
	\end{equation*}
	Thus we obtain by using \eqref{3.5.1} that 
	\begin{align*}
	\|A(f^{k}y)u\|_{f^{k+1}x_{0}} & \leq \|(A(f^{k}y)u)_H\|_{f^{k+1}x_{0}}+ \|(A(f^{k}y)u)_F\|_{f^{k+1}x_{0}} \\
	                              & \leq e^{\lambda_1+2\varepsilon}\|u_H\|_{f^{k}x_{0}}+ \theta_0 e^{\lambda_{j+1}+2\varepsilon}\|u_H\|_{f^{k}x_{0}}\\
	                              & \leq e^{\lambda_1+3\varepsilon}\|u\|_{f^{k}x_{0}}.
	\end{align*}
	Hence,  for any $u\in U_j(x_0,\theta_0)$,  by \eqref{3.0.1} and \eqref{2.6}, we conclude
	\begin{align*}
	\|\mathcal{A}_y^m(u)\|\leq \|\mathcal{A}_y^m(u)\|_{f^{m}x_{0}} \leq e^{(\lambda_1+3\varepsilon)m}\|u\|_{x_{0}}
	                                                             \leq le^{(\lambda_1+3\varepsilon)m}\|u\|
	                                                             \leq e^{(\lambda_1+4\varepsilon)m}\|u\|.
	\end{align*}
	This proves the conclusion	{\bf(\romannumeral1)}.
	
	{\bf(\romannumeral2)} For simplicity, we only prove $\mathcal{A}_{f^my}^{-m}V_j(x_{1},\theta_0)\subset V_j(x_{0},\eta\theta_0)$ and\\ $\|\mathcal{A}_y^m(v)\|\leq e^{(\lambda_{j+1}+4\varepsilon)m}\|v\|, \forall v\in \mathcal{A}_{f^my}^{-m}V_j(x_{1},\theta_0), 1\leq j\leq i.$ 
	
	Similar to the proof of \eqref{3.5.8}, one has 
	\begin{align*}
	\mathcal{A}_{f^ky}^{m-k}U_j(f^kx_{0},\theta) \subset U_j(x_{1},\eta\theta), ~\forall \theta_0\leq\theta\leq \eta_0^{-1}\theta_0^{-1},0\leq k\leq m-1.
	\end{align*}
	Let  $U_j^\circ(x,\theta):=\{u\in X:\|u_F\|_{x}< \theta \|u_H\|_{x}  \}$. Then
	\[U_j^\circ(x,\theta)=\bigcup_{n=1}^{\infty}U_j(x,\frac{n-1}{n}\theta).\]
	Therefore, for any $0\leq k\leq m-1,$
	\begin{align*}
	\mathcal{A}_{f^ky}^{m-k}U_j(f^kx_{0},\frac{n-1}{n}\eta^{-1}\theta_0^{-1}) \subset U_j(x_{1},\frac{n-1}{n}\theta_0^{-1}), ~\forall n>1,
	\end{align*}
	which implies $$ \mathcal{A}_{f^ky}^{m-k}U_j^\circ(f^kx_{0},\eta^{-1}\theta_0^{-1})\subset U_j^\circ(x_{1},\theta_0^{-1}).$$
	Since $A(x)$ is injective for any $x\in M$, we have
	$$U_j^\circ(f^kx_{0},\eta^{-1}\theta_0^{-1})\subset(\mathcal{A}_{f^ky}^{m-k})^{-1} U_j^\circ(x_{1},\theta_0^{-1}).$$
	Thus $$\mathcal{A}_{f^my}^{-m+k}\big(X\setminus U_j^\circ(x_{1},\theta_0^{-1})\big)\subset X\setminus U_j^\circ(f^kx_{0},\eta^{-1}\theta_0^{-1}),$$
	that is,
	$$\mathcal{A}_{f^my}^{-m+k}  V_j(x_{1},\theta_0)\subset V_j(f^kx_{0},\eta\theta_0), \forall 0\leq k\leq m-1.$$ 
	In particular, 
	\[\mathcal{A}_{f^my}^{-m}  V_j(x_{1},\theta_0)\subset V_j(x_{0},\eta\theta_0). \]
	Now, for any $w\in \mathcal{A}_{f^my}^{-m+k}  V_j(x_{1},\theta_0)\subset V_j(f^kx_{0},\eta\theta_0)\subset V_j(f^kx_{0},\theta_0),$ Let $w=w_H+w_F,$ where $w_H\in H_j(f^kx_{0}),w_F\in F_j(f^kx_{0})$. Then 
	\[\|A(f^kx_0)w_H\|_{f^{k+1}x_0}\leq e^{\lambda_1+\varepsilon}\|w_H\|_{f^{k}x_0},~\|A(f^kx_0)w_F\|_{f^{k+1}x_0}\leq e^{\lambda_{j+1}+\varepsilon}\|w_F\|_{f^{k}x_0}. \]	
	It follows that 	
	\begin{align*}
	\|A(f^kx_0)w\|_{f^{k+1}x_0} &\leq \|A(f^kx_0)w_H\|_{f^{k+1}x_0}+\|A(f^kx_0)w_F\|_{f^{k+1}x_0}\\
	                            &\leq e^{\lambda_1+\varepsilon}\theta_0\|w_F\|_{f^{k}x_0}+e^{\lambda_{j+1}+\varepsilon}\|w_F\|_{f^{k}x_0}\\
	                            &\leq e^{\lambda_{j+1}+2\varepsilon}\|w\|_{f^{k}x_0}.
	\end{align*}
	 Similar to \eqref{3.3}, we can obtain $\|A(f^ky)w-A(f^kx_0)w\|_{f^{k+1}x_0}\leq  c_0le^{\varepsilon}\rho^\alpha \|w\|_{f^{k}x_0}$. 
	 Therefore,	 for any $w\in \mathcal{A}_{f^my}^{-m+k}  V_j(x_{1},\theta_0),$ where $0\leq k\leq m-1,$ one has $A(f^ky)w\in \mathcal{A}_{f^my}^{-m+k+1}  V_j(x_{1},\theta_0)$ and
	\begin{align*}
	\|A(f^ky)w\|_{f^{k+1}x_0} &\leq \|A(f^ky)w-A(f^kx_0)w\|_{f^{k+1}x_0}+\|A(f^kx_0)w\|_{f^{k+1}x_0}\\
	                          &\leq c_0le^{\varepsilon}\rho^\alpha \|w\|_{f^{k}x_0}+e^{\lambda_{j+1}+2\varepsilon}\|w\|_{f^{k}x_0}\\
	                          &\leq e^{\lambda_{j+1}+3\varepsilon}\|w\|_{f^{k}x_0},
	\end{align*}
	if $\rho$ is small enough. It implies that for any $v\in \mathcal{A}_{f^my}^{-m}  V_j(x_{1},\theta_0)$,
	\[	\|A_{y}^{m}(v)\|\leq \|A_{y}^{m}(v)\|_{f^{m}x_0} \leq  e^{(\lambda_{j+1}+3\varepsilon)m}\|v\|_{x_0}\leq e^{(\lambda_{j+1}+4\varepsilon)m}\|v\|.  \]
	This completes the proof.
\end{proof}
	    
%%%%%%%%%%%%%%%%%%%%%%%   引理 3.2  %%%%%%%%%%%%%%%%%%%%%%%%%%   

   We now consider the cones for diffeomorphisms. For any $x\in \mathcal{R}^{Df}$, by subsection 2.4, we consider the trivialization $T_{\mathcal{N}(x)}M\equiv\mathcal{N}(x)\times\mathbb{R}^d.$ For any $y\in \mathcal{N}(x)$, we have the splitting $T_yM=E^u(y)\oplus E^s(y)$ which is  translated from $T_xM=E^u(x)\oplus E^s(x).$ For any ${u}\in T_yM$, let $u=u_u+u_s,$  where $u_u\in E^u(y),u_s\in E^s(y)$. We consider  the cones 
   \[U(y,\theta):=\{u\in T_yM:\|{u}_s\|^\prime_y\leq\theta\|{u}_u\|^\prime_y \},\] 
    \[V(y,\theta):=\{u\in T_yM:\|{u}_u\|^\prime_y\leq\theta\|{u}_s\|^\prime_y\}.\]
    Now for any $\rho$-pseudo-orbit $(x_n)_{n\in\mathbb{Z}}$ of $f^m$ in $\mathcal{R}^{Df}_\delta$  and for any $y\in M$ with $d(f^kx_{n},f^{k}(f^{nm}y))\leq \rho\cdot e^{-\frac{\chi}{2}\min\{k,m-k\}}$,~$k=0,\cdots,m$, we consider  the splitting $T_{f^{nm}y}M=E^u({f^{nm}y})\oplus E^s({f^{nm}y})$  translated from $T_{x_n}M=E^u({x_n})\oplus E^s({x_n}).$ Then we have the following Lemma, which  comes from Katok \cite{Katok}. We also give a proof here for the completeness.
\begin{lemma}\label{lemma 3.2}
	For any $0<\varepsilon<\varepsilon_i,$ $\delta>0,$ let $\theta=e^\varepsilon-1$, then there exist $\rho_1>0,$ $0<\eta^\prime<1$, such that for any $0<\rho<\rho_1$, any $\rho$-pseudo-orbit $(x_n)_{n\in\mathbb{Z}}$ of $f^m$ in $\mathcal{R}^{Df}_\delta$ with $m\geq\frac{\log l^\prime-\log c_1}{\varepsilon}, $  and for any $y\in M$ with $d(f^kx_{n},f^{k}(f^{nm}y))\leq \rho\cdot e^{-\frac{\chi}{2}\min\{k,m-k\}}$,~$k=0,\cdots,m$,  we have
	\[(D_{f^{nm}y}f^m)U(f^{nm}y,\theta)\subset U({f^{(n+1)m}y},\eta^\prime\theta),\] 
	\[(D_{f^{nm}y}f^{-m})V({f^{nm}y},\theta)\subset V({f^{(n-1)m}y},\eta^\prime\theta),\]
	for any $n\in \mathbb{Z}.$ Moreover, for any $u\in U({f^{nm}y},\theta), v\in V({f^{nm}y},\theta),$ we have
	\[\|D_{f^{nm}y}f^m(u)\|\geq e^{(\chi-4\varepsilon)m}\|u\|,\quad\|D_{f^{nm}y}f^{-m}(v)\|\geq e^{-(-\chi+4\varepsilon)m}\|v\|.\]
\end{lemma}
\begin{proof}
	We will only prove that $D_{y}f^mU({y},\theta)\subset U({f^{m}y},\eta^\prime\theta),$ and $\|D_{y}f^m(u)\|\geq e^{(\chi-4\varepsilon)m}\|u\|$ for $u\in U(y,\theta)$. the other conclusions can be proved analogously.
	
	Fix any $0<\varepsilon<\varepsilon_i,$ $\delta>0,$ let $\widetilde{\eta}^\prime=e^{-2\chi+4\varepsilon}$. Then we have the following claim.
	\begin{claim}
		There exists $0<\widetilde{\rho}_1<{l^\prime}^{-2},$ such that for any $0<\rho<\widetilde{\rho}_1$,  for any $x_{0},f^mx_{0}\in\mathcal{R}^{Df}_\delta$, $y\in M$ with  $d(f^kx_{0},f^{k}y)\leq \rho\cdot e^{-\frac{\chi}{2}\min\{k,m-k\}}$~for~$k=0,\cdots,m$,  we have 
		\[D_{f^ky}f {U}(f^ky,\theta)\subset {U}(f^{k+1}y,\widetilde{\eta}^\prime\theta),\quad \forall 0\leq k\leq m-2. \]
	\end{claim}
	\begin{proof}[Proof  of the Claim]
		For any $0\leq k\leq m-1$, $0<\rho<{l^\prime}^{-2}\leq l^\prime(x_0)^{-2}$, by Theorem \ref{thm 2.3}, 
		\begin{align*}
		d(f^kx_0,f^ky)\leq \rho\cdot e^{-\frac{\chi}{2}\min\{k,m-k\}}  & \leq l^\prime(x_0)^{-2}\cdot e^{-2\varepsilon\min\{k,m-k\}}\\
		                                                               & \leq l^\prime(f^kx_0)^{-2}\leq r(f^kx_0).
		\end{align*}
		Thus $f^ky\in \mathcal{N}(f^kx_0)$.
		Now for any  ${u}={u}_u+{u}_s\in {U}(f^ky,\theta),$ by subsection 2.4, considering  the identification of  $D_{f^kx_0}f=D_{0}(\widetilde{f}_{f^kx_0})$ and  $T_yM=E^u(y)\oplus E^s(y)$ with $\mathbb{R}^d=\widetilde{E}^u(f^kx_0)\oplus \widetilde{E}^s(f^kx_0)$, \eqref{2.10} and \eqref{2.11} give
		\begin{equation}\label{3.7}
		\|D_{f^kx_0}f({u}_u)\|^\prime_{f^{k+1}x_0}\geq e^{\chi-\varepsilon}\|{u}_u \|^\prime_{f^{k}x_0},\quad \|D_{f^kx_0}f({u}_s)\|^\prime_{f^{k+1}x_0}\leq e^{-\chi+\varepsilon}\|{u}_s \|^\prime_{f^{k}x_0}.
		\end{equation}
			Let ${w}=\big(D_{f^ky}{f}-D_{f^kx_0}{f}\big)({u})={w}_u+{w}_s$, where ${w}_u\in {E}^u(f^{k+1}y)$ and ${w}_s\in {E}^s(f^{k+1}y)$. 
		Since $\widetilde{E}^u(f^{k+1}x_0)$ and $\widetilde{E}^s(f^{k+1}x_0)$ are orthogonal, we have $\|{w}_u\|^\prime_{f^{k+1}y}\leq \|{w}\|^\prime_{f^{k+1}y}$.	
		 Then by Theorem \ref{thm 2.3},
		 \begin{equation}\label{3.8}
		 \begin{split}
		 \|{w}_u\|^\prime_{f^{k+1}y}\leq \|{w}\|^\prime_{f^{k+1}y} 
		 & \leq l^\prime(f^{k+1}x_0)\|D_{f^ky}{f}-D_{f^kx_0}{f}\|\cdot\|{u}\|\\
		 & \leq l^\prime e^{\varepsilon\min\{k+1,n-k-1\}}\cdot cd(f^kx_0,f^ky)^{r-1}\cdot\|{u}\|^\prime_{f^{k}y} \\
		 & \leq (1+\theta) l^\prime  e^\varepsilon c\rho^{r-1}\cdot e^{{(\varepsilon-\frac{\chi}{2}(r-1))}\min\{k,n-k\}}\|{u_u}\|^\prime_{f^{k}y}\\
		 & \leq (1+\theta) l^\prime  e^\varepsilon c\rho^{r-1}\|{u_u}\|^\prime_{f^{k}y}, 	 
		 \end{split}
		 \end{equation}
		since $\varepsilon< \frac{\chi(r-1)}{2r}$ and ${u}\in {U}(f^ky,\theta)$. Similarly,
		\begin{equation}\label{3.9}
		\|{w}_s\|^\prime_{f^{k+1}y} \leq (1+\theta) l^\prime  e^\varepsilon c\rho^{r-1}\|{u_u}\|^\prime_{f^{k}y}.
		\end{equation}
		Let
		\begin{align*}
		D_{{f^ky}}f({u})  =\big(D_{{f^ky}}f({u})\big)_u+\big(D_{{f^ky}}f({u})\big)_s
		                  \in {E}^u(f^{k+1}y)\oplus {E}^s(f^{k+1}y).
		\end{align*}
		Then by \eqref{2.9.1}, \eqref{3.7} and \eqref{3.8}
		\begin{equation}\label{3.10}
		\begin{split}
		\|\big(D_{{f^ky}}f({u})\big)_u\|^\prime_{f^{k+1}y} 
		& \geq \|\big(D_{{f^kx_0}}f({u})\big)_u\|^\prime_{f^{k+1}x_0}-\|{w}_u\|^\prime_{f^{k+1}y}\\
        & \geq e^{\chi-\varepsilon}\|{u}_u \|^\prime_{f^{k}x_0}-(1+\theta) l^\prime  e^\varepsilon c\rho^{r-1}\|{u_u}\|^\prime_{f^{k}y},\\
        & \geq e^{\chi-2\varepsilon}\|{u_u}\|^\prime_{f^{k}y},
		\end{split}
		\end{equation}
		if $\rho $ is taken small enough. Similarly, by \eqref{3.7} and \eqref{3.9},  
		 \begin{equation*}
		 \begin{split}
		 \|\big(D_{{f^ky}}f({u})\big)_s\|^\prime_{f^{k+1}y}
		 & \leq \|\big(D_{{f^kx_0}}f({u})\big)_s\|^\prime_{f^{k+1}x_0}+\|{w}_s\|^\prime_{f^{k+1}y}\\
		 & \leq e^{-\chi+\varepsilon}\|{u}_s \|^\prime_{f^{k}x_0}+(1+\theta) cl^\prime  e^\varepsilon\rho^{r-1}\|{u_u}\|^\prime_{f^{k}y},\\
		 & \leq \theta \cdot e^{-\chi+2\varepsilon}\|{u_u}\|^\prime_{f^{k}y},
		 \end{split}
		 \end{equation*}
		 if  $\rho $ is small enough. Thus $\|\big(D_{{f^ky}}f({u})\big)_s\|^\prime_{f^{k+1}y}\leq \widetilde{\eta}^\prime\theta \|\big(D_{{f^ky}}f({u})\big)_u\|^\prime_{f^{k+1}y}$, that is,
		 $D_{{f^ky}}f {U}(f^ky,\theta)\subset {U}(f^{k+1}y,\widetilde{\eta}^\prime\theta).$
		 This proves the claim.
	\end{proof}  
	The claim gives 
	$$D_yf^{m-1} U(y,\theta)\subset U({f^{m-1}y},\widetilde{\eta}^\prime\theta).$$
	Notice that $f^my\in\mathcal{N}(f^mx_0)$. Let $T_{f^{m}y}M=\widetilde{E}^u({f^{m}y})\oplus \widetilde{E}^s({f^{m}y})$ be translated from $T_{f^{m}x_0}M=E^u({f^{m}x_0})\oplus E^s({f^{m}x_0})$, $\|u\|_{f^my}^{\prime\prime}:=\|D_{f^my}(\Psi_{f^mx_0}^{-1})u\|$, and $\widetilde{U}(f^{m}y,\theta):=\{u\in T_{f^{m}y}M:\|\widetilde{u_s}\|^{\prime\prime}_{f^{m}y}\leq\theta\|\widetilde{u_u}\|^{\prime\prime}_{f^{m}y} \}$. Then the claim above also gives   
	$D_{f^{m-1}y}fU(f^{m-1}y,\theta)\subset \widetilde{U}({f^{m}y},\widetilde{\eta}^\prime\theta).$ Thus we have 
	\[D_yf^{m} U(y,\theta)\subset \widetilde{U}({f^{m}y},\widetilde{\eta}^\prime\theta). \]
	By the definition of $\mathcal{R}_\delta^{Df},$   $x\mapsto \Psi_{x}$  and the Oseledets splitting   $TM=E^u\oplus E^s$ are uniformly continuous on the compact set $\mathcal{R}_\delta^{Df}.$ Hence there exists $\widetilde{\eta}^\prime<{\eta}^\prime<1,$ such that
	\[\widetilde{U}({f^my},\widetilde{\eta}^\prime\theta)\subset U({f^my},{\eta}^\prime\theta),\quad \forall d(f^mx_{0},x_{1})\leq \rho,\]
	if $\rho$ is  small enough. Therefore, 
	$$D_{y}f^mU(y,\theta)\subset U({f^{m}y},\eta^\prime\theta).$$
	To show $\|D_{y}f^m(u)\|^\prime_{f^{m}y}\geq e^{(\chi-4\varepsilon)m}\|u\|^\prime_{y}$ for $u\in U(y,\theta)$, taking any   ${v}\in {U}(f^ky,\theta),$ $0\leq k\leq m-2,$ it follows from \eqref{3.10} that 
	\begin{equation*}
	\begin{split}
	\|D_{{f^ky}}f({v})\|^\prime_{f^{k+1}y}\geq \|\big(D_{f^kx}f({v})\big)_u\|^\prime_{f^{k+1}y} 
	& \geq e^{\chi-2\varepsilon}\|{v}_u\|^\prime_{f^{k}y} \\
	& \geq (1+\theta)^{-1}e^{\chi-2\varepsilon}\|{v}\|^\prime_{f^{k}y} \\
	& =    e^{\chi-3\varepsilon}\|{v}\|^\prime_{f^{k}y}.
	\end{split}
	\end{equation*}
	For  ${v}\in {U}(f^{m-1}y,\theta),$  the same reason gives $\|D_{{f^{m-1}y}}f({v})\|^{\prime\prime}_{f^{m}y}\geq  e^{\chi-3\varepsilon}\|{v}\|^\prime_{f^{m-1}y}.$
	Hence for any $u\in U(y,\theta)$, 
	$$\|D_{{y}}f^m({u})\|^{\prime\prime}_{f^{m}y}\geq e^{(\chi-3\varepsilon)m}\|{u}\|^\prime_{y}.$$
	Thus by \eqref{2.12} and $m\geq\frac{\log l^\prime-\log c_1}{\varepsilon},$ we conclude
	\begin{align*}
	\|Df_{y}^m(u)\|  \geq \frac{c_1}{l^\prime(f^mx_0)}e^{(\chi-3\varepsilon)m}\|u\|\geq e^{(\chi-4\varepsilon)m}\|u\|.
	\end{align*}
  This completes the proof of Lemma  \ref{lemma 3.2}.
\end{proof}
\subsection{Construction of hyperbolic horseshoes}%todo say something
   The aim of this subsection is to  construct a hyperbolic horseshoe $\Lambda$ satisfying the properties listed in Theorem \ref{thm A}.
 
   We begin with producing a separated set with sufficiently large cardinality. For $n\geq1,$ denote by $d_n(x,y)=\max_{0\leq k\leq n-1}d(f^kx,f^ky)$ the dynamical distance on $M$, and denote by $B_n(x,\rho)=\{y\in M: d_n(x,y)\leq \rho \}$ the  $d_n$-balls of radius $\rho$. Let $N_\mu(n,\rho,\bar{\delta})$ be the minimal numbers of $d_n$-balls of radius $\rho$ whose union has measure at least $\bar{\delta}.$ Then by \cite[Theorem 1.1]{Katok}, for any $\bar{\delta}>0,$
   \[h_\mu(f)=\lim\limits_{\rho\to 0}\liminf\limits_{n\to+\infty}\frac{1}{n}\log N_\mu(n,\rho,\bar{\delta}).\]
    Given any $0<\delta <1/2,$ let
   $\Lambda_\delta=\mathcal{R}^{Df}_\delta\cap \mathcal{R}^{\mathcal{A}}_\delta\cap\text{~supp}(\mu).$ Then $\mu(\Lambda_\delta)>1-2\delta>0.$ Take $\bar{\delta}=\frac{1}{2}\mu(\Lambda_\delta),$ then for any given $i\in\mathbb{N}_{k_0},$  $0<\varepsilon<\varepsilon_i$,  there exist $0<\rho_2<\varepsilon/2, N>1,$ such that for any $0<\rho<\rho_2,n\geq N,$ one has
   
   \begin{equation}\label{4.1}
   N_\mu(n,\rho,\bar{\delta})\geq e^{(h_\mu(f)-\varepsilon)n}.
   \end{equation} 
   Fix  a dense subset $\{\varphi_j\}_{j= 1}^\infty$ of the unit sphere of $C(M)$, then it induces a metric on the set of $f$-invariant measures $\mathcal{M}_f(M)$:  $$D(\mu_1,\mu_2)= \sum\limits_{j=1}^{\infty}\frac{|\int \varphi_jd\mu-\int \varphi_jd\nu|}{2^j},\quad \forall \mu_1,\mu_2\in \mathcal{M}_f(M).$$ 
   Take  $J$ large enough such that $\frac{1}{2^{J}}<\frac{\varepsilon}{8}$, and take  $\rho<\min\{\rho_0,\rho_1,\rho_2/2\}$ small enough (where $\rho_0,\rho_1$ are given by \ref{lemma 3.1}  and \ref{lemma 3.2} respectively), such that
   \begin{equation}\label{4.2}
   |\varphi_j(x)-\varphi_j(y)|\leq \varepsilon/4,\quad \forall~ d(x,y)\leq \rho,~ j=1,\cdots,J. 
   \end{equation}
   Since $\Lambda_\delta\subset \mathcal{R}^{Df}_\delta$, by Lemma \ref{thm 2.5}, there exists $0<\beta<\frac{1}{3}\rho$ and finite $(\rho, \beta,\frac{\chi}{2})$-rectangles $R(q_1),\cdots,R(q_t)$, such that $\cup_{j=1}^tB(q_j,\beta)\supset \Lambda_\delta$.
   We consider a partition $\mathcal{P}=\{P_1,\cdots,P_t \}$ of $\Lambda_\delta$, where  
   \[P_1=B(q_1,\beta)\cap\Lambda_\delta,\text{~and~} P_k=B(q_k,\beta)\cap\Lambda_\delta\setminus (\cup_{j=1}^{k-1}P_j), \forall 2\leq k\leq t.\]
   Let
   \begin{align*}
   \Lambda_{\delta,n}  :=  \Big\{  & x\in \Lambda_\delta: \text{~there exists~}  k\in[n, (1+\varepsilon)n]\text{~such that~} f^k(x)\in \mathcal{P}(x) \text{~and~}\\
   & \Big|\frac{1}{m}\sum\limits_{p=0}^{m-1}\varphi_j(f^px)-\int \varphi_j d\mu\Big|\leq \frac{\varepsilon}{4},~~
   \forall m\geq n, 1\leq j\leq J \Big\}.
   \end{align*}
   \begin{claim}
   	$\lim\limits_{n\to\infty}\mu(\Lambda_{\delta,n})=\mu(\Lambda_{\delta}). $
   \end{claim}
   \begin{proof}[Proof of the Claim.]
   	Let
   	$$A_n =\left\{ x\in \Lambda_\delta: \text{~there exists~}  k\in[n, (1+\varepsilon)n]\text{~such that~}  f^k(x)\in \mathcal{P}(x)\right\},$$
   	and  $A_{n,j} =\{ x\in P_j: \text{~there exists~}  k\in[n, (1+\varepsilon)n]\text{~such that~}  f^k(x)\in P_j\}.$  Then $A_n=\cup_{j=1}^tA_{n,j}$. 
   	Considering  $P_1$, we may assume $\mu(P_1)>0$. For any $\tau>0,$ let
   	\[A_{n,1}^{\tau}=\left\{x\in P_1: \mu(P_1)-\tau\leq\frac{1}{m}\sum_{j=0}^{m-1}\chi_{P_1}(f^jx)\leq \mu(P_1)+\tau, \forall m\geq n   \right\}. \]   
   	Then Birkhoff Ergodic Theorem gives $\mu(\cup_{n\geq 1}A_{n,1}^{\tau})=\mu(P_1)$. Since $A_{1,1}^{\tau}\subset A_{2,1}^{\tau}\subset\cdots,$ we have
   	\[\lim\limits_{n\to\infty}\mu(A_{n,1}^{\tau})=\mu(P_1). \]
   	Take $\tau<\frac{\varepsilon}{2+\varepsilon}\mu(P_1)$, then  for any $x\in A_{n,1}^{\tau}$,  by the definition of $ A_{n,1}^{\tau}$, 
   	\begin{align*}
   	\text{card}\{k\in [n,n+n\varepsilon]:f^k(x)\in P_1 \} & \geq (\mu(P_1)-\tau)(n+n\varepsilon)-(\mu(P_1)+\tau)n\\
   	& = n(\mu(P_1)\varepsilon-2\tau-\tau\varepsilon)\\
   	& \geq 1,
   	\end{align*}
   	if n is taken large enough. Thus $x\in A_{n,1}$, that is $A_{n,1}^{\tau}\subset A_{n,1}$. Therefore, $\lim\limits_{n\to\infty}\mu(A_{n,1})=\mu(P_1).$ Reproduce the proof above for every $P_j$,  then we conclude
   	\begin{equation}\label{4.3}
   	\lim\limits_{n\to\infty}\mu(A_n)=\mu(\Lambda_{\delta}).
   	\end{equation}
   	Let
   	\[B_n=\Big\{x\in \Lambda_\delta:\Big|\frac{1}{m}\sum\limits_{k=0}^{m-1}\varphi_j(f^kx)-\int \varphi_j d\mu\Big|\leq \frac{\varepsilon}{4}, \forall m\geq n, 1\leq j\leq J \Big \}.\]
   	Then by Birkhoff Ergodic Theorem,  $\mu(\cup_{n\geq 1}B_n)=\mu(\Lambda_{\delta})$. Since $B_1\subset B_2\subset\cdots,$ one has
   	\[\lim\limits_{n\to\infty}\mu(B_n)=\mu(\Lambda_{\delta}). \]  
   	Together with \eqref{4.3}, 
   	\[\lim\limits_{n\to\infty}\mu(\Lambda_{\delta,n})=\lim\limits_{n\to\infty}\mu(A_n\cap B_n) =\mu(\Lambda_{\delta}).\]
   	This proves the Claim.	
   \end{proof}
   
   Choose $n>\max\{\frac{\log l}{\varepsilon},\frac{\log l^\prime-\log c_1}{\varepsilon},N, \frac{1}{\varepsilon}\log t \}$ large enough such that  $\mu(\Lambda_{\delta,n})>\frac{1}{2}\mu(\Lambda_{\delta})=\bar{\delta},$  and $n\varepsilon+1<e^{\varepsilon n}$.   Denote by $E$  an $(n,2\rho)$-separated set of $\Lambda_{\delta,n}$ of maximum cardinality.  Then $\cup_{x\in E}B_n(x,2\rho)\supset \Lambda_{\delta,n}.$ By \eqref{4.1},
   \[ \text{card}(E)\geq N_\mu(n,2\rho,\mu(\Lambda_{\delta,n}))\geq N_\mu(n, 2\rho,\bar{\delta})\geq e^{(h_\mu(f)-\varepsilon)n}. \]
   For $ k\in [n, n(1+\varepsilon)]$, let $F_k=\{x\in E:f^k(x)\in \mathcal{P}(x)\}$. And take $m\in [n, n(1+\varepsilon)]$ satisfying $\text{card}(F_m)=\max \{\text{card}(F_k):n\leq k\leq n(1+\varepsilon) \}.$ Then
   \[\text{card}(F_m)\geq \frac{1}{n\varepsilon+1}\text{card}(E)\geq  \frac{1}{n\varepsilon+1}e^{(h_\mu(f)-\varepsilon)n}\geq  e^{(h_\mu(f)-2\varepsilon)n}.\]
   Choose $P\in \mathcal{P}$ satisfying $\text{card}(F_m\cap P)=\max \{\text{card}(F_m\cap P_k):1\leq k\leq t\}.$ Then
   \[\text{card}(F_m\cap P)\geq \frac{1}{t}\text{card}(F_m)\geq \frac{1}{t}e^{(h_\mu(f)-2\varepsilon)n}.\]
  By getting rid of some points in $F_m\cap P$, we may assume
   \begin{equation}\label{4.4}
   \frac{1}{t}e^{(h_\mu(f)-2\varepsilon)n}\leq \text{card}(F_m\cap P)\leq\frac{1}{t}e^{(h_\mu(f)+2\varepsilon)n}.
   \end{equation}

   By the definition of the  partition $\mathcal{P}$ , there exists $q\in \{q_1,\cdots,q_t\}$, such that $P\subset B(q,\beta)\cap \Lambda_\delta$. Thus for any $x\in F_m\cap P,$ since $x,f^m(x)\in B(q,\beta)\cap \Lambda_\delta,$ by Definition \ref{def 2.4},  the connected component $C\big(x,R(q)\cap f^{-m}R(q)\big)$ of $R(q)\cap f^{-m}R(q)$ containing $x$ is an admissible $s$-rectangle in $R(q)$, and $f^mC\big(x,R(q)\cap f^{-m}R(q)\big)$ is an admissible $u$-rectangle in $R(q)$.     
   
   We claim that if $x_1,x_2\in F_m\cap P$ with $x_1\neq x_2,$ then $C\big(x_1,R(q)\cap f^{-m}R(q)\big)\cap C\big(x_2,R(q)\cap f^{-m}R(q)\big)=\varnothing$. Indeed, if there is
   $y\in C\big(x_1,R(q)\cap f^{-m}R(q)\big)\cap C\big(x_2,R(q)\cap f^{-m}R(q)\big),$
    by Definition \ref{def 2.4}, one sees $d(f^kx_j,f^ky)\leq \rho$  for any $0\leq k\leq m$ and $j=1,2$.
   Thus $d_m(x_1,x_2)\leq 2\rho$. However,  since $F_m\cap P$ is an $(n,2\rho)$-separated set, we obtain $d_m(x_1,x_2)\geq d_n(x_1,x_2)>2\rho$, which is a contradiction. Therefore, there are at least   $\text{card}( F_m\cap P)$ disjoint s-rectangles in $R(q)$, mapped by $f^m$ to $\text{card}(F_m\cap P)$  disjoint  admissible u-rectangles in $R(q)$.

   Let 
   \[\Lambda^*=\bigcap\limits_{n\in\mathbb{Z}}f^{-mn}\Big(\bigcup\limits_{x\in F_m\cap P}C\big(x,R(q)\cap f^{-m}R(q)\big) \Big).\]
   Then $f^m|_{\Lambda^*}$ is conjugate to a full shift in $\text{card}(F_m\cap P)$-symbols. And for any $y\in \Lambda^*,$ any $n\in \mathbb{Z}$, there exists $x_n\in F_m\cap P$ such that $f^{mn}(y)\in C\big(x_n,R(q)\cap f^{-m}R(q)\big).$ It follows that $y$ and $(x_n)_{n\in\mathbb{Z}}$ satisfy the conditions of  Lemma \ref{lemma 3.2}.
   Thus by Lemma \ref{lemma 3.2} and \cite[Theorem 6.1.2]{barreira000pesin}, $\Lambda^*$ is hyperbolic for $f^m.$ Let 
   \[\Lambda=\Lambda^*\cup f(\Lambda^*)\cup\cdots\cup f^{m-1}(\Lambda^*).\]
   Then $\Lambda$ is a hyperbolic horseshoe.
   
    It remains to show the conclusions $\text{(\romannumeral1)}-\text{(\romannumeral4)}$ of Theorem \ref{thm A} hold  for this $\Lambda$. 
   
   $\text{(\romannumeral1)}$ Since
   $$h_{\text {\em top}}(f|_{\Lambda})=\frac{1}{m}h_{\text {\em top}}(f^m|_{\Lambda^*})=\frac{1}{m}\log \text{card}(F_m\cap P),$$
   by \eqref{4.4},
   $h_{\text {\em top}}(f|_{\Lambda})\geq -\frac{1}{m}\log t +  \frac{n}{m}(h_\mu(f)-2\varepsilon).$
   Since  $\frac{1}{\varepsilon}\log t< n\leq m\leq n(1+\varepsilon)$, 
   \[h_{\text {\em top}}(f|_{\Lambda})\geq -\varepsilon+\frac{1}{1+\varepsilon}(h_\mu(f)-2\varepsilon)\geq h_\mu(f)-( h_\mu(f)+3)\varepsilon.\]
   By \eqref{4.4}, we also have
   \[ h_{\text {\em top}}(f|_{\Lambda})\leq -\frac{1}{m}\log t +  \frac{n}{m}(h_\mu(f)+2\varepsilon)\leq h_\mu(f)+2\varepsilon\leq h_\mu(f)+( h_\mu(f)+3)\varepsilon.\]

   $\text{(\romannumeral2)}$ By the construction of $\Lambda$ and the definition of  $R(q)$, for any $y\in \Lambda,$ there exist $x\in F_m\cap P, 0\leq k\leq m-1$  such that  $d(y,f^kx)\leq\rho<\varepsilon/2.$ Since $x\in \Lambda_\delta\subset\text{~supp}(\mu),$ We conclude $\Lambda$ is contained in an  $\varepsilon/2$-neighborhood of  supp$(\mu)$.
   
   $\text{(\romannumeral3)}$  For any $f$-invariant measure $\nu$ supported on $\Lambda$, we may assume $\nu$ is ergodic first. Since 
   $$D(\mu,\nu)\leq \sum\limits_{j=1}^{J}\frac{|\int \varphi_jd\mu-\int \varphi_jd\nu|}{2^j}+\frac{1}{2^{J-1}}\leq \sum\limits_{j=1}^{J}\frac{|\int \varphi_jd\mu-\int \varphi_jd\nu|}{2^j}+\frac{\varepsilon}{4},$$
    it's enough to show: $|\int \varphi_jd\mu-\int \varphi_jd\nu|\leq \frac{3}{4}\varepsilon$, $\forall 1\leq j\leq J.$ Take $y\in \Lambda^*$ and $s\in \mathbb{N}$  large enough such that
   \[\Big|\frac{1}{ms}\sum\limits_{k=0}^{ms-1}\varphi_j(f^ky)-\int \varphi_j d\nu\Big|\leq \frac{\varepsilon}{4},~ 1\leq j\leq J.\]
   Then there exist $x_0,x_1,\cdots,x_{s-1}\in F_m\cap P$ such that  
   $$d(f^{mk+t}y,f^tx_k)\leq \rho,\quad\forall 0\leq k \leq s-1, 0\leq t\leq m-1.$$ 
   By \eqref{4.2} and the construction of $\Lambda_{\delta,m}$, we obtain  $|\int \varphi_jd\mu-\int \varphi_jd\nu|\leq \frac{3}{4}\varepsilon$, $1\leq j\leq J,$  which implies $D(\mu,\nu)\leq \varepsilon$.
   
   If $\nu$ is not ergodic, by the ergodic decomposition theorem, $\nu$-almost every ergodic component is supported on $\Lambda$. Hence
   \begin{align*}
   D(\mu,\nu)  = \sum\limits_{j=1}^{\infty}\frac{|\int \varphi_jd\mu-\int \varphi_jd\nu|}{2^j} 
   & =  \sum\limits_{j=1}^{\infty}\frac{|\int(\int \varphi_jd\mu-\int \varphi_jd\nu_x)d\nu(x)|}{2^j} \\
   &\leq   \int D(\mu,\nu_x) d\nu(x)  \\
   & \leq \varepsilon.
   \end{align*}
              
\subsection{Dominated splitting for cocycles } The  conclusion $\text{(\romannumeral4)}$ of Theorem \ref{thm A} is contained in the following Proposition.

%%%%%%%%%%%%%%%%%%%  Prop 3.3 %%%%%%%%%%%%%%%%%
%todo  continuous invariant splliting
\begin{proposition}\label{prop 3.3}
   Under the condition of Theorem \ref{thm A}, for any  $i\in\mathbb{N}_{k_0},$  any $0<\varepsilon<\varepsilon_i$, let $\Lambda$ be constructed as above. Then there exists a continuous $\mathcal{A}$-invariant splitting on $\Lambda$ 
   \[X=E_1(y)\oplus\cdots\oplus E_i(y)\oplus F_i(y),\]
   with~$\dim(E_j)=d_j$ for $1\leq j\leq i$, such that for any $y\in \Lambda$, one has 
   \[e^{(\lambda_j-4\varepsilon)m}\|u\|\leq \|\mathcal{A}_y^m(u)\|\leq e^{(\lambda_j+4\varepsilon)m}\|u\|,\quad \forall u\in E_j(y), 1\leq j\leq i, \]
   \[\|\mathcal{A}_y^m(v)\|\leq e^{(\lambda_{i+1}+4\varepsilon)m}\|v\|,\quad \forall v\in F_i(y). \]
\end{proposition}
   In order to prove this Proposition, we require the following definitions. Let $\mathcal{G}(X)$ denote the Grassmannian of closed subspaces of $X$, endowed with the Hausdorff metric $d_H$, defined by: 
   \[d_H(E,F) = \max\{\sup\limits_{u\in S_E}dist(u,S_F),\sup\limits_{v\in S_F}dist(v,S_E) \},\quad\forall E,F\in \mathcal{G}(\mathfrak{X}),\]
   where $S_F=\{v\in F:\|v\|=1 \}$, $dist(u,S_F)=\inf\{\|u-v\|:v\in S_F\}$. Denote by $\mathcal{G}_j({X})$,  $\mathcal{G}^j({X})$  the  Grassmannian of $j$-dimensional and j-codimensional closed subspaces, respectively. Then by \cite[chapter \uppercase\expandafter{\romannumeral4}, \S2.1 ]{Kato},  $(\mathcal{G}({X}),d_H)$ is a complete metric space, and  $\mathcal{G}_j({X})$ , $\mathcal{G}^j({X})$ are closed in $\mathcal{G}({X})$. In order to compute $d_H$ conveniently, we introduce the \emph{gap} $\hat{\delta}$, defined by 
   \[ \hat{\delta}(E,F) = \max\{\sup\limits_{u\in S_E}dist(u,F),\sup\limits_{v\in S_F}dist(v,E) \},\quad\forall E,F\in \mathcal{G}({X}).\]
  Note that the gap  is not a metric on $\mathcal{G}({X})$,  but 
   $$\hat{\delta}(E,F)\leq d_H(E,F)\leq 2\hat{\delta}(E,F).$$
   See \cite[chapter \uppercase\expandafter{\romannumeral4}, \S2.1 ]{Kato} for a proof.
   
   We now prove Proposition \ref{prop 3.3}.
   
%\begin{lemma}\label{lemma 3.4}
% 	Under the condition of Theorem \ref{thm A}, for any  $i\in\mathbb{N}_{k_0},$   $0<\varepsilon<\varepsilon_i$ and any $1\leq j\leq i,$  there exists a continuous sub-bundle $H_j$ on $\Lambda^*$ such that for any  $y\in \Lambda^*$,
% 	\begin{enumerate}
% 		\item  $\dim(H_j(y))=D_j$, where $D_j=d_1+\cdots+d_j.$
% 		\item $\mathcal{A}(y,m)H_j(y)=H_j(f^m(y))$.
% 	    \item $e^{(\lambda_j-4\varepsilon)m}\|u\|\leq \|\mathcal{A}(y,m)u\|\leq e^{(\lambda_1+4\varepsilon)m}\|u\|,\quad \forall u\in H_j(y)$. 
% 	\end{enumerate}
%\end{lemma}
   \begin{proof}
   	 Fix any $i\in\mathbb{N}_{k_0}$, $0<\varepsilon<\varepsilon_i$ and $1\leq j\leq i$. We first assume $y\in \Lambda^*,$  then  by the definition of $\Lambda^*$, there exist $\{x_n\}_{n\in \mathbb{Z}}\subset F_m$,  such that  
   	 $$d(f^kx_{n},f^{k}(f^{nm}y))\leq \rho\cdot e^{-\frac{\chi}{2}\min\{k,m-k\}}, ~\forall n\in \mathbb{Z},k=0,\cdots,m.$$ 

   	We first prove $\{\mathcal{A}_{f^{-mn}y}^{mn}H_j(x_{-n})\}_{n\geq 1}$ is a Cauchy sequences in $\mathcal{G}({X})$.
   	Denote $\widetilde{H}_j(x_{-n})=\mathcal{A}_{f^{-m(n+1)}y}^{m}H_j(x_{-n-1}).$ Then
   	$$\mathcal{A}_{f^{-mn}y}^{mn}\widetilde{H}_j(x_{-n})=\mathcal{A}_{f^{-m(n+1)}y}^{m(n+1)}H_j(x_{-n-1}).$$
    By Lemma  \ref{lemma 3.1}, $\mathcal{A}_{f^{-mn}y}^{mn}\widetilde{H}_j(x_{-n})\subset U_j(x_{0},\theta_0)$. Since $A(x)$ is injective for every $x\in M$, we have  \begin{equation}\label{3.18}
    {X}=\mathcal{A}_{f^{-mn}y}^{mn}\widetilde{H}_j(x_{-n})\oplus F_j(x_0).
    \end{equation} 	
   	Now for any $u\in H_j(x_{-n})$, 
   	\begin{equation}\label{3.19}
   	\begin{split}
   	  & dist\left(\frac{\mathcal{A}_{f^{-mn}y}^{mn}(u)}{\|\mathcal{A}_{f^{-mn}y}^{mn}(u)\|},\mathcal{A}_{f^{-m(n+1)}y}^{m(n+1)}H_j(x_{-n-1})\right)\\
   	= & ~ \inf_{v\in \widetilde{H}_j(x_{-n}) }\left\{\big\|\frac{\mathcal{A}_{f^{-mn}y}^{mn}(u)}{\|\mathcal{A}_{f^{-mn}y}^{mn}(u)\|}- \mathcal{A}_{f^{-mn}y}^{mn}(v)\big\| \right\}\\
   	= & ~ \frac{\|u\|}{\|\mathcal{A}_{f^{-mn}y}^{mn}(u)\|}\inf_{v\in \widetilde{H}_j(x_{-n})  }\left\{\|\mathcal{A}_{f^{-mn}y}^{mn}(\frac{u}{\|u\|}- v)\| \right\}.   	
   	\end{split}
   	\end{equation}
   	By \eqref{3.18},  for  $\mathcal{A}_{f^{-mn}y}^{mn}\frac{u}{\|u\|}\in {X}$, there exists $v_n\in \widetilde{H}_j(x_{-n})$ such that
   	\begin{equation*}
   	\mathcal{A}_{f^{-mn}y}^{mn}(\frac{u}{\|u\|})-\mathcal{A}_{f^{-mn}y}^{mn}(v_n)\in  F_j(x_0)\subset V_j(x_0,\theta_0).
   	\end{equation*}  
   	Then it follows from \eqref{3.19} and Lemma  \ref{lemma 3.1}  that
  \begin{align*}
   	      & dist\left(\frac{\mathcal{A}_{f^{-mn}y}^{mn}(u)}{\|\mathcal{A}_{f^{-mn}y}^{mn}(u)\|},\mathcal{A}_{f^{-m(n+1)}y}^{m(n+1)}H_j(x_{-n-1})\right)\\
   	 \leq &~ \frac{\|u\|}{\|\mathcal{A}_{f^{-mn}y}^{mn}(u)\|}\cdot\|\mathcal{A}_{f^{-mn}y}^{mn}(\frac{u}{\|u\|}- v_n)\|\\
     \leq & ~ e^{(\lambda_{j+1}-\lambda_j+8\varepsilon)mn}\cdot \|\frac{u}{\|u\|}- v_n\|.
  \end{align*}  
  	\begin{claim}
  		There exists  $c=c(\delta)>0$, such that 	$\|\frac{u}{\|u\|}- v_n\|\leq c.$
  	\end{claim}
  	\begin{proof}[Proof of the Claim]
  	Using \eqref{3.18},  we obtain 
  	\begin{equation}\label{3.20}
  	{X}=\widetilde{H}_j(x_{-n})\oplus \mathcal{A}_y^{-mn}F_j(x_0). 
  	\end{equation}
  	Indeed, define $T:X\to X$ by $T=\big(\mathcal{A}_{f^{-mn}y}^{mn}|_{\widetilde{H}_j(x_{-n})} \big)^{-1}\circ \pi_1\circ\mathcal{A}_{f^{-mn}y}^{mn}$, where $\pi_1$ is the  projection associated with \eqref{3.18} onto $\mathcal{A}_{f^{-mn}y}^{mn}\widetilde{H}_j(x_{-n})$ parallel to $F_j(x_0)$. Then T  has image $\widetilde{H}_j(x_{-n})$ and kernel $\mathcal{A}_y^{-mn}F_j(x_0)$, which implies \eqref{3.20}.
  	
  	Let $\pi_2$ be the  projection operator associated with \eqref{3.20} onto  $\mathcal{A}_y^{-mn}F_j(x_0)$ parallel to $\widetilde{H}_j(x_{-n})$ . Then 
  	$$\frac{u}{\|u\|}- v_n=\pi_2(\frac{u}{\|u\|})\in \mathcal{A}_y^{-mn}F_j(x_0).$$
  	 Since the splitting ${X}=H_j(x)\oplus F_j(x)$ is uniformly continuous on $\mathcal{R}^\mathcal{A}_\delta$, there exists $c=c(\delta)>0$ such that for any $x\in\mathcal{R}^\mathcal{A}_\delta$ and any $H\subset U_j(x,\theta_0), F\subset V_j(x,\theta_0)$ with $\mathfrak{X}=H\oplus F$, one has  	
  	\begin{equation}\label{3.21}
  	 \|\pi^H(x)\|\leq c,\|\pi^F(x)\|\leq c,
  	\end{equation}
  	  where $\pi^H,$ $\pi^F$ are the  projections   associated with the splitting  ${X}=H_j(x)\oplus F_j(x)$.  Therefore, 
  	\[\|\frac{u}{\|u\|}- v_n\|=\|\pi_2(\frac{u}{\|u\|})\|\leq c. \]
  	This proves the claim.
  \end{proof}
  	The Claim gives
   	\begin{align*}
   	      dist\left(\frac{\mathcal{A}_{f^{-mn}y}^{mn}(u)}{\|\mathcal{A}_{f^{-mn}y}^{mn}(u)\|},\mathcal{A}_{f^{-m(n+1)}y}^{m(n+1)}H_j(x_{-n-1})\right)
   	\leq  e^{(\lambda_{j+1}-\lambda_j+8\varepsilon)mn}\cdot c.
   	\end{align*}  
   	Similarly,	 for any $v\in H_j(x_{-n-1})$,  we have
   	\begin{align*}
      	 dist\left(\frac{\mathcal{A}_{f^{-m(n+1)}y}^{m(n+1)}(v)}{\|\mathcal{A}_{f^{-m(n+1)}y}^{m(n+1)}(v)\|},\mathcal{A}_{f^{-mn}y}^{mn}H_j(x_{-n})\right)
   	\leq  e^{(\lambda_{j+1}-\lambda_j+8\varepsilon)mn}\cdot c. 
   	\end{align*}
   	Thus
   	 \begin{equation*}
   	 \begin{split}
   	      \hat{\delta}\big(\mathcal{A}_{f^{-mn}y}^{mn}H_j(x_{-n}),\mathcal{A}_{f^{-m(n+1)}y}^{m(n+1)}H_j(x_{-n-1})\big)
   	\leq  e^{(\lambda_{j+1}-\lambda_j+8\varepsilon)mn}\cdot c,
   	\end{split}
   	 \end{equation*}
   	which implies  $\{\mathcal{A}_{f^{-mn}y}^{mn}H_j(x_{-n})\}_{n\geq 1}$ is a Cauchy sequence in  $(\mathcal{G}(X),d_H)$. Similarly,  $\{\mathcal{A}_{f^{mn}y}^{-mn}F_j(x_{n})\}_{n\geq 1}$ is a  Cauchy sequence in  $(\mathcal{G}(X),d_H)$.
   	Let 
   \begin{equation*}
   	H_j(y):=\lim\limits_{n\to\infty} \mathcal{A}_{f^{-mn}y}^{mn}H_j(x_{-n}),  \quad
   F_j(y):=\lim\limits_{n\to\infty}\mathcal{A}_{f^{mn}y}^{-mn}F_j(x_{n}). 
   \end{equation*}
   Since $X= \mathcal{A}_y^{mn}H_j(x_{0})\oplus F_j(x_{n})$, similar to the proof of \eqref{3.20}, one sees  $X= H_j(x_{0})\oplus\mathcal{A}_{f^{mn}y}^{-mn} F_j(x_{n})$. Thus 
   $$\text{ codim}\big(\mathcal{A}_{f^{mn}y}^{-mn}F_j(x_{n})\big)=D_j=\dim\big(\mathcal{A}_{f^{-mn}y}^{mn}H_j(x_{-n})\big), $$
   where $D_j=d_1+\cdots+d_j$. It follows $H_j(y)\in \mathcal{G}_{D_j}(X),F_j(y)\in \mathcal{G}^{D_j}(X)$, since $\mathcal{G}_{D_j}(X)$ and $\mathcal{G}^{D_j}(X)$ are closed subset of  $\mathcal{G}(X)$.  Notice that $H_j(y)\subset U_j(x_0,\theta_0),$ and $ F_j(y)\subset \mathcal{A}_{f^{m}y}^{-m}V_j(x_1,\theta_0)\subset V_j(x_0,\theta_0)$, we conclude
   $$ X= H_j(y)\oplus F_j(y).$$
   And by Lemma \ref{lemma 3.1}, we have 
    \begin{equation*}
    e^{(\lambda_j-4\varepsilon)m}\|u\|\leq \|\mathcal{A}_y^m(u)\|\leq e^{(\lambda_1+4\varepsilon)m}\|u\|,\quad \forall u\in H_j(y), 
    \end{equation*}
  \begin{equation*}
   \|\mathcal{A}_y^m(v)\|\leq e^{(\lambda_{j+1}+4\varepsilon)m}\|v\|,\quad \forall v\in F_j(y). 
  \end{equation*}
   
   In general, for any $z\in \Lambda,$ there exists $0\leq k\leq m-1,y\in \Lambda^*$ such that $z=f^ky.$ similar to the proof above, we can also get
   \begin{equation*}
   H_j(z):=\lim\limits_{n\to\infty} \mathcal{A}_{f^{-mn}(f^ky)}^{mn}H_j(f^kx_{-n}),  
   \end{equation*}
   \begin{equation*}
   F_j(z):=\lim\limits_{n\to\infty}\mathcal{A}_{f^{mn}(f^ky)}^{-mn}F_j(f^kx_{n}),
   \end{equation*} 
   satisfying $ X= H_j(z)\oplus F_j(z),$ $H_j(z)\in \mathcal{G}_{D_j}(X),F_j(z)\in \mathcal{G}^{D_j}(X)$, and
   	\begin{equation}\label{3.22}
   	e^{(\lambda_j-4\varepsilon)m}\|u\|\leq \|\mathcal{A}_z^m(u)\|\leq e^{(\lambda_1+4\varepsilon)m}\|u\|,\quad \forall u\in H_j(z), 
   \end{equation}
   \begin{equation}\label{3.23}
    \|\mathcal{A}_z^m(v)\|\leq e^{(\lambda_{j+1}+4\varepsilon)m}\|v\|,\quad \forall v\in F_j(z). 
   \end{equation}
\begin{claim}
      	The splitting  $ X= H_j(z)\oplus F_j(z)$ is $\mathcal{A}$-invariant on $\Lambda$. 
\end{claim}   
  	\begin{proof}[Proof of the claim]
  	 Before proving the invariance of the splitting, we show that: for any $y\in \Lambda^*,0\leq k\leq m-1$, any $H_{n,k}\in \mathcal{G}_{D_j}(X)$  satisfying $H_{n,k} \subset U_j(f^kx_{-n},\theta_0)$ , one has 
  	 \[H_j(f^ky)=\lim\limits_{n\to\infty} \mathcal{A}_{f^{-mn}(f^ky)}^{mn}{H_{n,k}}.\]
%  	 If $k=m-1$, by the proof of Lemma \ref{lemma 3.1}, $H_j(f^mx_{-n})\subset V_j(x_{-n+1},\theta_0)$,   therefore
   Indeed, since $X=\mathcal{A}_{f^{-mn}(f^ky)}^{mn}H_{n,k}\oplus F_j(x_0)$. Thus for any $u\in H_j(f^kx_{-n})$, similar to the estimation above, we have
  	 \begin{equation*}
  	 \begin{split}
  	 & dist\left(\frac{\mathcal{A}_{f^{-mn}(f^ky)}^{mn}(u)}{\|\mathcal{A}_{f^{-mn}(f^ky)}^{mn}(u)\|},\mathcal{A}_{f^{-mn}(f^ky)}^{mn}H_{n,k}\right)\\
  	 = & ~ \frac{\|u\|}{\|\mathcal{A}_{f^{-mn}(f^ky)}^{mn}(u)\|}\cdot\inf_{v\in H_{n,k} }\|\mathcal{A}_{f^{-mn}(f^ky)}^{mn}(\frac{u}{\|u\|}- v)\| \\
  	 \leq &~   e^{(\lambda_{j+1}-\lambda_j+8\varepsilon)mn}\cdot c,
  	 \end{split}
  	 \end{equation*}
  	which implies
  	 $d_H\left(\mathcal{A}_{f^{-mn}(f^ky)}^{mn}H_j(x_{-n}),\mathcal{A}_{f^{-mn}(f^ky)}^{mn}{H_{n,k}}\right)\to 0.$
  	 Hence,
  	 \begin{equation}\label{3.25}
  	 H_j(f^ky)=\lim\limits_{n\to\infty} \mathcal{A}_{f^{-mn}(f^ky)}^{mn}{H_{n,k}}, \quad\forall 0\leq k\leq m-1.
  	 \end{equation}
  	  Now for any $z=f^ky\in \Lambda,$ where $y\in\Lambda^*,0\leq k\leq m-1$. 
  	  %  	 \[F_j(f^ky)=\lim\limits_{n\to\infty} \mathcal{A}(f^{mn}(f^ky),-mn)F_j(f^kx_{n}), \forall 0\leq k\leq m. \]
    If $0\leq k\leq m-2$, then by \eqref{3.0.1},  $A(f^{-mn}z)H_j(f^kx_{-n})\in U_j(f^{k+1}x_{-n},\theta_0)$; if $k=m-1$, then  by \eqref{3.5.7}, $A(f^{-mn}z)H_j(f^{m-1}x_{-n})\in U_j(x_{-n+1},\theta_0)$, thus \eqref{3.25} gives
    \[H_j(fz)=\lim\limits_{n\to\infty} \mathcal{A}_{f^{-mn}(fz)}^{mn}\left(A(f^{-mn}z)H_j(f^kx_{-n})\right).  \] 
  	 Therefore, for any $u\in H_j(z),$
  	 \begin{align*}
  	 &dist\left(A(z)u,\mathcal{A}_{f^{-mn}(fz)}^{mn}A(f^{-mn}z)H_j(f^kx_{-n})  \right)\\
  	 \leq &~\|A(z)\|\cdot dist\left(u,\mathcal{A}_{f^{-mn}(z)}^{mn}H_j(f^kx_{-n})  \right)\\
  	 \leq &~\|A(z)\|\cdot d_H\left(H_j(z),\mathcal{A}_{f^{-mn}(z)}^{mn}H_j(f^kx_{-n})  \right)\to 0,
  	 \end{align*}
  	 which implies $A(z)u\in H_j(fz)$. Thus $A(z)H_j(z)\subset H_j(fz)$. It can be proved analogously that  $A(z)F_j(z)\subset F_j(fz)$. Since $\dim(A(z)H_j(z))= \dim(H_j(fz))$, we have $A(z)H_j(z)=H_j(fz)$. 
% Therefore, $X=A(z)H_j(z)\oplus  F_j(fz)$. Similar to the proof of \eqref{3.20}, one sees $X=H_j(z)\oplus A(z)^{-1} F_j(fz)$, which implies 
%  	 $$\text{codim}( F_j(z))=D_j=\text{codim}(A(z)^{-1} F_j(fz)).$$
%  	 The injectivity of $A(z)$ and the equation $A(z)F_j(z)\subset F_j(fz)$ give $F_j(z)\subset A(z)^{-1}F_j(fz)$. Thus  $F_j(z)= A(z)^{-1}F_j(fz)$.  	   	 	 
  	   	\end{proof}
 \begin{claim}
   The splitting  $ X= H_j(z)\oplus F_j(z)$ is continuous on $\Lambda$. 
\end{claim} 
    \begin{proof}[Proof of the Claim]
       Let $\pi^{H_j}_z$, $\pi^{H_j}_z$ be the projection operators associated with  $ X= H_j(z)\oplus F_j(z)$. Then for any $\tilde{z}\in \Lambda,u\in H_j(\tilde{z})$, let $w=\mathcal{A}_{\tilde{z}}^{-mn}(u)\in H_j(f^{-mn}\tilde{z})$, by the invariance of the splitting and \eqref{3.21}, \eqref{3.22}, \eqref{3.23}, we may estimate
       \begin{align*}
       \|\pi^{F_j}_z\mathcal{A}_{f^{-mn}{z}}^{mn}(w)\|& = \|\mathcal{A}_{f^{-mn}{z}}^{mn}\pi^{F_j}_{f^{-mn}{z}}(w)\|\\
                                                 & \leq e^{(\lambda_{j+1}+4\varepsilon)mn}\|\pi^{F_j}_{f^{-mn}{z}}(w)\|\\
                                                 & \leq c\cdot e^{(\lambda_{j+1}+4\varepsilon)mn}\|w\|\\
                                                 & \leq  c\cdot e^{(\lambda_{j+1}-\lambda_{j}+8\varepsilon)mn}\|u\|.
       \end{align*}
       Thus 
       	\begin{align*}
       \|\pi^{F_j}_z(u)\| & = \|\pi^{F_j}_z\mathcal{A}_{f^{-mn}\tilde{z}}^{mn}(w)\| \\
       & \leq \|\pi^{F_j}_z\left(\mathcal{A}_{f^{-mn}\tilde{z}}^{mn}-\mathcal{A}_{f^{-mn}{z}}^{mn}\right)(w)\|+\|\pi^{F_j}_z\mathcal{A}_{f^{-mn}{z}}^{mn}(w)\| \\
       &\leq c\cdot \left(e^{(4\varepsilon-\lambda_j)mn}\|\mathcal{A}_{f^{-mn}\tilde{z}}^{mn}-\mathcal{A}_{f^{-mn}{z}}^{mn}\|+e^{(\lambda_{j+1}-\lambda_{j}+8\varepsilon)mn}\right)\|u\|,
       \end{align*}
       Which gives 
       $$\|\pi^{F_j}_z|_{H_j(\tilde{z})}\|\leq c\cdot e^{(4\varepsilon-\lambda_j)mn}\|\mathcal{A}_{f^{-mn}\tilde{z}}^{mn}-\mathcal{A}_{f^{-mn}{z}}^{mn}\|+c\cdot e^{(\lambda_{j+1}-\lambda_j+8\varepsilon)mn} .$$
       	Similarly, for $v\in F_j(\tilde{z})$,             
   	\begin{align*}
   	e^{(\lambda_j-4\varepsilon)mn}\|\pi^{H_j}_z(v)\| & \leq \|\mathcal{A}_z^{mn}\circ\pi^{H_j}_z(v)\| \\
   	                                               & = \|\pi^{H_j}_{f^{mn}z}\circ\mathcal{A}_z^{mn}(v)\| \\
   	                                               & \leq \|\pi^{H_j}_{f^{mn}z}\|\cdot\big(\|\mathcal{A}_z^{mn}(v)-\mathcal{A}_{\tilde{z}}^{mn}(v)\|+\|\mathcal{A}_{\tilde{z}}^{mn}(v)\| \big) \\
   	                                               &\leq c\cdot\big(\|\mathcal{A}_z^{mn}-\mathcal{A}_{\tilde{z}}^{mn}\|+e^{(\lambda_{j+1}+4\varepsilon)mn} \big)\|v\|.
   	\end{align*}
   	Therefore, 
   	$$\|\pi^{H_j}_z|_{F_j(\tilde{z})}\|\leq c\cdot e^{(4\varepsilon-\lambda_j)mn}\|\mathcal{A}_z^{mn}-\mathcal{A}_{\tilde{z}}^{mn}\|+c\cdot e^{(\lambda_{j+1}-\lambda_j+8\varepsilon)mn} .$$
   	
   Now for any $\tau>0,$ take $n$ large enough such that $c\cdot e^{(\lambda_{j+1}-\lambda_j+8\varepsilon)mn}<\tau,$ and then take $\delta>0$ small enough such that for any $\tilde{z}\in B(z,\delta),$
    $$c\cdot e^{(4\varepsilon-\lambda_j)mn}\cdot\max\left\{\|\mathcal{A}_{f^{-mn}\tilde{z}}^{mn}-\mathcal{A}_{f^{-mn}{z}}^{mn}\|,\|\mathcal{A}_z^{mn}-\mathcal{A}_{\tilde{z}}^{mn}\|\right\}<\tau.$$ Then we have $\|\pi^{H_j}_z|_{F_j(\tilde{z})}\|\leq 2 \tau,\|\pi^{F_j}_z|_{H_j(\tilde{z})}\leq 2\tau.$
   	Thus by \cite[Remark 10]{BM2015}, 
   	$$d_H\big(F_j(z),F_j(\tilde{z})\big)\leq 4D_j\|\pi^{H_j}_z|_{F_j(\tilde{z})}\|\leq 8D_j\tau,$$ 
   	$$d_H\big(H_j(z),H_j(\tilde{z})\big)\leq 4D_j\|\pi^{F_j}_z|_{H_j(\tilde{z})}\|\leq 8D_j\tau,$$
   	which gives the continuity of $F_j(z)$ and $H_j(z)$.
   \end{proof}
   At last, for any $z\in\Lambda,1\leq j\leq i$, let $E_j(z):=H_j(z)\cap F_{j-1}(z)$, where $F_0(z):={X}$. Then $\dim(E_j)=d_j$. Since  
   \[A(z)E_j(z)\subset A(z)H_j(z)\cap A(z)F_{j-1}(z)\subset H_j(fz)\cap F_{j-1}(fz)=E_j(fz), \]
   and $\dim\left(A(z)E_j(z)\right)=d_j=\dim(E_j(fz))$, we have $A(z)E_j(z)=E_j(fz).$
     Therefore, we obtain a continuous $\mathcal{A}$-invariant splitting on $\Lambda$:
   \[{X}=E_1(z)\oplus\cdots\oplus E_i(z)\oplus F_i(z),\]
   	with~$\dim(E_j)=d_j, \forall 1\leq j\leq i$. And by \eqref{3.22},  \eqref{3.23}, we conclude
   	\[e^{(\lambda_j-4\varepsilon)m}\|u\|\leq \|\mathcal{A}_z^{m}(u)\|\leq e^{(\lambda_j+4\varepsilon)m}\|u\|,\quad \forall u\in E_j(z), 1\leq j\leq i, \]
   	\[\|\mathcal{A}_z^{m}(v)\|\leq e^{(\lambda_{i+1}+4\varepsilon)m}\|v\|,\quad \forall v\in F_i(z). \]
   	This completes the proof.
   \end{proof}

 \bibliographystyle{amsplain}

\end{document}